\documentclass[12pt]{article}
\usepackage{amsmath}
\usepackage{float}
\usepackage{amsthm}
\usepackage{amssymb}
\usepackage{graphics}
\usepackage{graphicx}
\usepackage{enumerate} 
\usepackage{multicol}
\usepackage{lmodern}
\usepackage{marvosym} 
\usepackage{lipsum}
\usepackage{mwe}
\usepackage{caption}
\usepackage{subcaption} 
\newtheoremstyle{case}{}{}{}{}{}{:}{ }{}
\theoremstyle{case}

\newcommand{\be}{\begin{equation}}
\newcommand{\ee}{\end{equation}}
\newcommand{\ben}{\begin{eqnarray*}}
\newcommand{\een}{\end{eqnarray*}}
\newtheorem{examp}{\sc Example}
\newtheorem{remk}{\sc Remark}
\newtheorem{corol}{\sc Corollary}
\newtheorem{lemma}{\sc Lemma}
\newtheorem{theorem}{\sc Theorem}
\newtheorem{defn}{\sc Definition}
\newcommand{\bt}{\begin{theorem}}
\newcommand{\et}{\end{theorem}}
\newcommand{\bl}{\begin{lemma}}
\newcommand{\el}{\end{lemma}}
\newcommand{\bed}{\begin{defn}}
\newcommand{\eed}{\end{defn}}
\newcommand{\brem}{\begin{remk}}
\newcommand{\erem}{\end{remk}}
\newcommand{\bex}{\begin{examp}}
\newcommand{\eex}{\end{examp}}
\newcommand{\bcl}{\begin{corol}}
\newcommand{\ecl}{\end{corol}}

\newcommand{\NI}{\noindent}

\newcommand{\al}{\alpha}

\newcommand{\fal}{\forall}
\newcommand{\raro}{\rightarrow}

\newcommand{\vsp}{\vskip 0.5em}

\newcommand{\lam}{\lambda}

\theoremstyle{definition}
\theoremstyle{remark}

\numberwithin{equation}{section}
\numberwithin{theorem}{section}
\numberwithin{lemma}{section}

\begin{document}
\title{Bounded Homotopy Path Approach to Find the Solution of Linear Complementarity Problems}
\author{ A. Dutta$^{a, 1}$, A. K. Das$^{b, 2}$, R. Jana$^{b, 3}$\\
\emph{\small $^{a}$Department of Mathematics, Jadavpur University, Kolkata, 700 032, India}\\	
\emph{\small $^{b}$SQC \& OR Unit, Indian Statistical Institute, Kolkata, 700 108, India}\\
\emph{\small $^{1}$Email: aritradutta001@gmail.com}\\
\emph{\small $^{2}$Email: akdas@isical.ac.in}\\
\emph{\small $^{3}$Email: rwitamjanaju@gmail.com} \\
 }
\date{}
\maketitle

\date{}
\maketitle
\begin{abstract}
	 \NI In this article, we introduce a new homotopy function to trace the trajectory by applying modified homotopy continuation method for finding the solution of the linear complementarity problem.  Earlier several authors attempted to propose  homotopy functions based on original problems. We propose the homotopy function based on the Karush-Kuhn-Tucker condition of the corresponding quadratic programming problem. The proposed approach extends the processability of  the larger class of linear complementarity problem and overcomes the limitations of other existing homotopy approaches.  We show that the homotopy path approaching the solution is smooth and bounded with positive tangent direction of the homotopy path.  Various classes of numerical examples are illustrated to show the effectiveness of the proposed algorithm and the superiority of the algorithm among other existing iterative methods.\\

\NI{\bf Keywords:} Linear complementarity problem, homotopy method, interior point method, strictly feasible point. \\

\NI{\bf AMS subject classifications:} 90C33, 15A39, 15B99, 14F35.
\end{abstract}
\footnotetext[1] {Corresponding author}
\footnotetext[2] {The author R.Jana presently working in an integrated steel plant of India}

\section{Introduction}
Eaves and Saigal \cite{eaves1972homotopies} formed an important class of globally convergent methods for solving systems of non-linear equations, which is known as homotopy method. Such methods have been used to constructively prove the existence of solutions to many economic and engineering problems. Let $X,Y$ be two topologocal spaces and  $f, g:X \to Y$ be continuous maps. A homotopy from $f$ to $g$ is a continuous function $H:X \times [0,1] $$ \to Y$ satisfying $H(x,0) = f(x),$ $H(x,1) = g(x)  \ \forall x \in X.$ If such a homotopy exists, then $f$ is homotopic to $g $ and it is denoted by $f \simeq g.$ Let $f,g:R\to R$ any two continuous, real functions, then $f \simeq g.$ Now we define a function $H:R\times [0,1] \to R $ by $H(x,t)=(1-t)f(x)+tg(x).$ Clearly $H$ is continuous and $H(x,0)=f(x),$ $H(x,1)=g(x).$ Thus $H$ is a homotopy between $f$ and $g.$ Let $X,Y$ be two topological spaces and  Map$(X,Y)$ be the set of all continuous maps from $X$ to $Y.$ Homotopy is an equivalence relation on Map$(X,Y).$
\vsp
 The fundamental idea of the homotopy continuation method is to solve a problem by tracing a certain continuous path that leads to a solution to the problem. Thus, defining a homotopy mapping that yields a finite continuation path plays an essential role in a homotopy continuation method. The homotopy method \cite{watson1989globally} is itself an important class of globally convergent methods. Many homotopy methods are proposed for constructive proof of the existence of solutions to systems of nonlinear equations, nonlinear optimization problems, Brouwer fixed point problems, nonlinear programming, game problem and complementarity problems \cite{watson1989modern}. Chen et al. \cite{chen2016computing} proposed a homotopy algorithm for computing complex eigenpairs of a tensor in a tensor complementarity problem. Han \cite{han2017homotopy} proposed a homotopy method for finding the unique positive solution to a multilinear system with a nonsingular $M$-tensor and a positive right side vector.
\vsp
The linear complementarity problem is well studied in the literature on mathematical
programming and arises in a number of applications in operations research, control theory,
mathematical economics, geometry and engineering. For recent works on this problem and
applications see \cite{das2017finiteness}, \cite{article12}, \cite{article11} and \cite{article03} and references therein. In complementarity theory several matrix classes are considered due to the study of theoretical properties, applications and its solution methods. For details see \cite{jana2019hidden}, \cite{jana2021more}, \cite{article1}, \cite{mohan2001more}, \cite{neogy2013weak} and \cite{neogy2005almost} and references cited therein. The problem of computing the value vector and optimal stationary strategies for structured stochastic games is formulated as a linear complementary problem for discounted and undiscounded zero-sum games. For details see \cite{mondal2016discounted}, \cite{neogy2008mixture} and \cite{neogy2005linear}. The complementarity problem establishes an important connections with multiobjective programming problem for KKT point and the solution point \cite{article78}. The complementarity problems are considered with respect to principal pivot transforms and pivotal method to its solution point of view. For details see \cite{das2016properties}, \cite{neogy2012generalized} and \cite{neogy2005principal}.

We are interested in solving the complementarity problem, mainly the linear complementarity problem. The linear complementarity problem is identified as an important mathematical programming problem and provides a unifying framework for several optimization problems like linear programming, linear fractional programming, convex quadratic programming and the bimatrix game problem. The linear complementarity problem arising from a free boundary problem can be
 reformulated as a fixed-point equation. Zhang \cite{zhang2021modified} presented a modified modulus-based multigrid method to solve this fixed-point equation. 
The concept of complementarity is synonymous with the notion of system equilibrium.
Among the many facets of research in linear complementarity problems, the area that has received thorough attention in recent years is the development of robust and efficient algorithms for solving various kinds of linear complementarity problems. Kojima et al. showed that the interior point method for linear programming problem was a kind of path-following method. This polynomial time-bound method is widely used to solve LCP$(q, A)$, but some matrices are not processable by this method as well as by Lemke's algorithm. For details see \cite{jana2018processability}  Modulus based algorithm is one of the proposed iterative method to solve linear complementarity problem. Van Bokhoven proved that the modulus algorithm works when the matrix involved is a symmetric P-matrix. Kappel et al.\cite{kappel1986iterative} 
extended van Bokhoven's results by showing that the modulus 
algorithm can be applied to a class of non-symmetric P-matrices. Schafer\cite{schafer2004modulus} showed the convergency of the modulus algorithm for three subclasses of $P$-matrices. Hadjidimos et al. \cite{hadjidimos2009nonstationary}, \cite{hadjidimos2012iterative} proposed a new method, the scaled extrapolated block modulus algorithm, as well as an improved version of the very recently introduced modulus-based matrix splitting modified AOR iteration method to find the solution of thelinear complementarity problem with $H_+$-matrix. Zheng et al. \cite{zheng2013accelerated},\cite{zheng2014convergence}, \cite{zheng2017relaxation} showed that  for the large sparse linear complementarity problem, established a relaxation modulus-based matrix splitting iteration method, a class of accelerated
modulus-based matrix splitting iteration methods by reformulating it
as a general implicit fixed-point equation, which covers the known modulus-based
matrix splitting iteration methods and presented the convergence conditions when 
the matrix involved is either a positive definite matrix or an $H_+$-matrix. Dai et al.\cite{dai2019preconditioned} proposed a preconditioned two-step modulus-based matrix splitting iteration method for linear complementarity problems associated with an $M$-matrix. For further details see \cite{bai1999convergence}, \cite{cui2021relaxation}, \cite{dong2009modified}, \cite{liu2016general}, \cite{chen2016computing}, \cite{article3} and \cite{jana2018semimonotone}.

 In the literature it was proved that the homotopy method converges globally to the solution of LCP$(q, A),$ where $A$ is a positive semidefinite matrix \cite{yu2006combined}, a $P$-matrix \cite{xuuu}, an $N$-matrix \cite{N} or a $P_*$-matrix \cite{Wang} with respect to different type of homotopy functions. Han\cite{han2017homotopy}, \cite{han2019continuation} introduced a Kojima–Megiddo–Mizuno type continuation method for solving tensor complementarity problems. He showed that there exists a bounded continuation trajectory when the tensor is strictly semi-positive and any limit point tracing the trajectory gives a solution of the tensor complementarity problem. Moreover, when the tensor is strong strictly semi-positive, tracing the trajectory will converge to the unique solution. In this paper, we attempt to introduce another homotopy function and condition for global convergence of the homotopy method to solve LCP$(q, A),$ where  $A$ belongs to various matrix classes.
\vsp
The paper is organized as follows. Section 2 presents some basic notations and results. In section 3, we propose a new homotopy function to find the solution of LCP$(q, A)$. We construct a smooth and bounded homotopy path under some conditions to find the solution of the linear complementarity problem as the homotopy parameter $\lambda$ tends to $0$. We prove an if and only if condition to get the solution of LCP$(q, A)$ from the solution of the homotopy equation. We also find the sign of the positive tangent direction of the homotopy path. We use a modified interior-point bounded homotopy path algorithm for solving the linear complementarity problem in section 4. Finally, in section 4, we consider various matrix classes namely, PSD, $N$, almost $C_0,$ singular $Q_0$, $Q$, ${E_0}^s$, almost $\bar{N}$-matrix, $N_0$-matrix of exact order $2$ and $\bar{N}$-matrix of exact order $2.$  Many of these classes are not processable by Lemke's algorithm,   existing homotopy methods and  modulus based method. We consider these classes to show the effectiveness of the homotopy function.

\section{Preliminaries}
\noindent We denote the $n$ dimensional real space by $R^n$ where $R^n_+$ and $R^{n}_{++}$ denote the nonnegative and positive orthant of $R^n.$ We consider vectors and matrices with real entries. Any vector $x\in R^{n}$ is a column vector and  $x^{t}$ denotes the row transpose of $x.$ $e$ denotes the vector of all $1.$ If $A$ is a matrix of order $n,$ $\al \subseteq \{1, 2, \cdots, n\}$ and $\bar{\al} \subseteq \{1, 2, \cdots, n\} \setminus \al$ then $A_{\al \bar{\al}}$ denotes the submatrix of $A$ consisting of only the rows and columns of $A$ whose indices are in $\al$ and $\bar{\al}$ respectively. $A_{\al \al}$ is called a principal submatrix of A and det$(A_{\al \al})$ is called a principal minor of $A.$ We define $\mathcal{F}=\{x\in R^n:x>0,Ax+q>0\}, \ 
\mathcal{\bar{F}}=\{x \in R^n:x\geq 0, Ax+q \geq 0\},
\mathcal{F}_1=\mathcal{F} \times R_{++}^n \times R_{++}^n$ and 
$\mathcal{\bar{F}}_1=\mathcal{\bar{F}} \times R_{+}^n \times R_{+}^n.$ $\partial{\mathcal{F}_1}$ denotes the boundary of $\bar{\mathcal{F}_1}.$

 The linear complementarity problem \cite{neogy2005principal} is defined as follows: 

Given square matrix $A\in R^{n\times n}$ and a vector $\,q\,\in\,R^{n},\,$ the linear complementarity problem is to find $w \in R^n$ and $x \in R^n$ such that
\begin{equation}\label{1}
w - Ax = q, w \geq 0, \, x \geq 0,
\end{equation}
\begin{equation} \label{2}
x^tw = 0.
\end{equation}
This problem is denoted as LCP$(q, A).$
Several applications of linear complementarity problems are reported in operations research \cite{pang1995complementarity}, multiple objective programming problems \cite{kostreva1993linear}, mathematical economics and engineering. For details see \cite{ferris1997engineering}, \cite{mohan2001more}, \cite{neogy2006some}, \cite{jana2019hidden} and \cite{jana2021more}. 

A matrix $A\in R^{n\times n}$ is said to be a/an \\
\NI $-$ {\it positive semidefinite} (PSD) matrix if $x^{t}Ax\geq 0,\;\fal\;x\in R^{n}.$ \\
\NI $-$ {\it $P_0(P)$}-matrix if all its principal minors are  nonnegative(positive).\\
\NI $-$ {\it $N$}-matrix  if all its principal minors are  negative.\\
\NI $-$ {\it $P_*$}-matrix if $\exists$ a constant $\tau > 0$ such that for any $x \in R^n,$
$$(1 + \tau)\sum_{i \in I_+(x)} x_i(Mx)_i + \sum_{i \in I_{-}(x)} x_i(Mx)_i \geq 0$$ where $I_{+}(x) = \{i \in N: x_i(Mx)_i > 0\}$ and $I_{-}(x) = \{i \in N: x_i(Mx)_i \leq 0\}.$ \\
$-$ $Z$-matrix if off-diagonal elements are all non-positive and $K\,(K_0)$-matrix if it is a $Z$-matrix as well as $P\,(P_0)$-matrix. ($K$-matrix is also known as $M$-matrix).\\
\NI $-$ {\it copositive} $(C_{0})$ matrix if $x^{t}Ax\geq 0,\;\fal\;x\geq 0.$ \\
\NI $-$ {\it almost $C_0$}-matrix if it is copositive of up to order $n-1$ but not of order $n.$ \\
\NI $-$ {\it  $N_0$}-matrix if $\det A_{\alpha \alpha} \leq 0 (< 0) \ \forall\ \alpha \subseteq \{1,2,\cdots,n\}.$\\
\NI $-$ {\it  almost $N0(N)$}-matrix if $\det A_{\alpha \alpha} \leq 0 (< 0) \ \forall\ \alpha \subset \{1,2,\cdots,n\}$ and $\det A > 0$.\\
\NI $-$ {\it $N_0$-matrix of exact order} $k \, (1 \leq k \leq n)$ if every principal submatrix of order $(n-k)$ is an $N_0$-matrix and every principal minor of order $r,$ $(n-k) < r \leq n$ is positive.\\
\NI $-$ {\it $\bar{N}$}-matrix \cite{mohan1992} if there exists a sequence $\{A^{(k)}\}$ where $A^{(k)} = [a_{ij}^{(k)}]$ are $N$-matrices such that $a_{ij}^{(k)} \raro a_{ij}$ for all $i, j \in \{1, 2, \cdots n\}.$\\
\NI $-$ {\it $Q$}-matrix if for every $q\in R^{n},$ LCP$(q,A)$ has a solution. \\
\NI $-$ {\it $Q_{0}$}-matrix if for any $q\in R^{n},$ (1.1) has a solution implies that LCP$(q, A)$ has a solution. \\
\NI $-$ {\it ${E_{0}}^s$}-matrix if $x^TAx=0, Ax \geq 0, x\geq 0 \implies A^Tx \leq 0.$\\
\NI $-$ {\it nondegenerate} matrix if all principal minors of the matrix $A$ are nonzero.\\
For further details about matrix classes see \cite{mohan1992}, \cite{neogy2005principal}, \cite{neogy2005almost},  \cite{dutta2021column}, \cite{dutta2021some}, \cite{neogy2005linear}, \cite{doi:10.1137/040613585}.
\vsp
The basic idea of homotopy methods can be explained as to construct a homotopy from the auxiliary mapping $g$ to the object mapping $p.$ The original problem can be solved by following the homotopy path from the zero set of the auxiliary mapping $g$ to the zero set of the object mapping $p.$ The difficulty of finding a strictly feasible initial point for the interior point algorithm can be avoided by combining the interior point with the homotopy method. Furthermore, the global convergence of the homotopy methods can guarantee the global convergence for the combined homotopy interior point methods. Suppose the given problem is to find a root of the non-linear equation $p(x) = 0$ and suppose $g(x) = 0$ is auxiliary function with an unique solution $x_0.$ Then the homotopy equation can be written as $H(x, \lambda)= \lambda g(x) + (1-\lambda)p(x), \, 0 \leq \lambda \leq 1.$ Then we consider $H(x, \lambda) = 0.$ The value of $\lambda$ will start from $1$ and goes to $0.$ In this way one can find the solution of the given equation $p(x) = 0$ from the solution of $g(x) = 0.$ 
\vsp
The key idea to solve LCP$(q,A)$ by the homotopy method is to solve a system of equations of the form $H(x,\lambda)=0,$ where $H:R^n\times [0,1] \to R^n, x\in R^n, \lambda \in [0,1]$ is called homotopy parameter. The homotopy method aims to trace out entire path of equilibria in $H^{-1}=\{(x,\lambda): H(x,\lambda)=0\}$ by varrying both $x$ and $\lambda.$ Now we define a parametric path as a set of functions $(x(s),\lambda(s))\in H^{-1}.$ When we move along the homotopy path, the auxiliary variable $s$ either decreases or increases monotonically. Differentiating $H(x(s),\lambda(s))=0$ with respect to $s$ we get $\frac{\partial H}{\partial x}x'(s)+\frac{\partial H}{\partial \lambda}\lambda'(s)=0,$ where $\frac{\partial H}{\partial x}$ and $\frac{\partial H}{\partial \lambda} $ are $n\times n$ jacobian matrix of $H$ and $n\times 1$ column vector respectively. So this is a system of $n$ differential equations in $n+1$ unknowns ${x_i}'(s) \ \forall \ i$ and $\lambda'(s).$ this system of differential equations has many solutions, which differ by monotone transformation of the auxiliary variable  $s.$
\vsp
Now we state some results which will be required in the next section.
\begin{lemma}\cite{cottle2009linear} \label{p01}
Let $M$ be a $P_0$-matrix. Then for each vector $z\neq 0$, there exists an index $i$ such that $z_i\neq 0$ and $z_i(Mz)_i \geq 0.$
\end{lemma}
\begin{lemma}\cite{cottle2009linear} \label{p02}
If $M$ is a $P_0$ matrix, then $M^t$ is also $P_0$.
\end{lemma}
\begin{lemma} \cite{chow1978finding} \label{main}
Let $U \subset R^n$ be an open set and $f :R^n \to R^p$ be smooth. We say $y \in R^p$ is a regular value for $f$ if $\text{Range} \, Df(x) = R^p $ $\forall x \in f^{-1}(y),$ where $Df(x)$ denotes the $n \times p$ matrix of partial derivatives of $f(x).$
\end{lemma}
\begin{lemma}\label{par} \cite{Wang}
Let $V \subset R^n, U \subset R^m$ be open sets, and let $\phi:V\times U \to R^k$ be a $C^\alpha$ mapping, where $\alpha >\text{max}\{0,m-k\}.$ If $0\in R^k$ is a regular value of $\phi,$ then for almost all $a \in V, 0$ is a regular value of $\phi _ a=\phi(a,.).$    
\end{lemma}
\begin{lemma}\label{inv} \cite{Wang}
	Let $\phi : U \subset R^n \to R^p$ be $C^\alpha$ mapping, where $\alpha >\text{max}\{0,n-p\}.$ Then $\phi^{-1}(0)$ consists of some $(n-p)$ dimensional $C^\alpha$ manifolds. 
\end{lemma}
\begin{lemma}\label{cl} \cite{N} One-dimensional smooth manifold is diffeomorphic to a unit circle or a unit interval.
\end{lemma}

\section{Main results}
We first discuss some existing homotopy functions.Watson \cite{watson1974variational} illustrated an outline of homotopy approach for complementarity problem. Chow et al. \cite{chow1978finding} developed sufficiently powerful theoretical tools for homotopy methods. In 2006, Yu et al. \cite{yu2006combined} proposed the following homotopy function to solve the LCP$(q, A)$ where $A$ is a \textit{positive semidefinite matrix,}
\begin{equation}\label{psdyu}
H(w,w^{(0)}, \lambda)=\left[\begin{array}{c} 
(1-\lambda)[Ax+q-y] + \lambda(x-x^{(0)}) \\
XYe-\lambda e\\
\end{array}\right]=0.
\end{equation}

Zhao et al. \cite{N} proposed the following homotopy function in 2010 to solve LCP$(q, A)$ where $A$ is an $N$-matrix, 
\begin{equation}\label{zhaon}
H(w,w^{(0)}, \lambda)=\left[\begin{array}{c} 
(1-\lambda)[y-Ax-q] + \lambda(x-x^{(0)}) \\
Xy - \lambda X^{(0)}y^{(0)}\\
\end{array}\right]=0.
\end{equation}  

Later Xu et al. \cite{xuuu} developed another homotopy function for finding the solution of LCP$(q, A)$ where $A$ is a $P$-matrix, 
\begin{equation}\label{xup}
H(w,w^{(0)}, \lambda)=\left[\begin{array}{c} 
(1 - \lambda)[y-Ax-q] - \lambda(x-x^{(0)}) \\
Xy - \lambda X^{(0)}y^{(0)}\\
\end{array}\right]=0.
\end{equation}

Wang et al. \cite{5609642} showed that linear complementarity problem with $P_{*}$-matrix can be solved using the homotopy function
\begin{equation}\label{wangp}
H(w,w^{(0)}, \lambda)=\left[\begin{array}{c} 
(1 - \lambda)[Ax+q]- y + \lambda y^{(0)} \\
Xy - \lambda X^{(0)}y^{(0)}\\
\end{array}\right]=0.
\end{equation}

We propose a new homotopy function to solve LCP$(q, A)$ based on the KKT condition.\\ 
\begin{equation} \label{homf}
H(y,y^{(0)},\lambda)=\left[\begin{array}{c} 
	(1-\lambda)[(A+A^t)x+q-z_1-A^tz_2]+\lambda(x-x^{(0)}) \\
	Z_1x-\lambda Z_1^{(0)}x^{(0)}\\
	Z_2(Ax+q)-\lambda Z_2^{(0)}(Ax^{(0)} + q)\\
\end{array}\right]=0
\end{equation}\\
where $Z_1=\text{diag}(z_1),$ $Z_2=\text{diag}(z_2),$ $Z_1^{(0)}=\text{diag}(z_1^{(0)}),$ $Z_2^{(0)}=\text{diag}(z_2^{(0)}),$ $y=(x,z_1,z_2) \in R_+^n \times R_+^n \times R_+^n,$ $y^{(0)}=(x^{(0)},{z_1}^{(0)},{z_2}^{(0)})\in \mathcal{F}_1,$ and $\lam \in (0,1].$ We denote $\Gamma_y^{(0)}=\{(y,\lam)\in R^{3n}\times (0,1]: H(y,y^{(0)},\lam)=0\} \subset \mathcal{F}_1 \times (0,1]\}.$\\

Here $\lambda$ varies from $1$ to $0,$ and starting from $\lambda =1$ to $\lambda \to 0$ if we get a smooth  bounded curve, then we will get a finite solution of the homotopy equation \ref{homf} at $\lambda \to 0.$ 
At $\lambda \to 1,$ the homotopy equation \ref{homf} gives the solution $(y^{(0)},1),$ and at $\lambda \to 0,$ the homotopy equation \ref{homf} gives the solution of the system of following equations:
\begin{center}
	$(A+A^t)x+q-z_1-A^tz_2=0$\\
	$Z_1x=0$\\
	$Z_2(Ax+q)=0$\\
\end{center}
where $Z_1=\text{diag}(z_1)$ and $Z_2=\text{diag}(z_2).$ 

Let $z_{1{I_1}} = 0$ and $x_{I_2}=0,$ where $I_1 \cup I_2=\{n\}.$ Let $z_{2J_1}=0$ and $ (Ax+q)_{J_2}=0,$ where $J_1 \cup J_2=\{n\}.$ If the solution of the homotopy function \ref{homf}, $y=(x,z_1,z_2)$ gives the solution of LCP$(q,A)$ which is $x,$ then  $x_{{I_2}^c}\neq0$ $\implies$ ${(Ax+q)_{{I_2}^c}}=0.$ This implies that ${{I_2}^c}\subseteq J_2$ and  ${(Ax+q)_{{J_2}^c}}\neq0 \implies x_{{J_2}^c}=0,$ which implies that  ${{J_2}^c}\subseteq I_2.$  ${{I_2}^c}=J_2,$  ${{J_2}^c}=I_2$ give the nondegenerate solution of LCP$(q,A)$ and ${{I_2}^c}\subset J_2,$  ${{J_2}^c}\subset I_2$ give the degenerate solution of LCP$(q,A).$  When $I_1 \cap I_2 = \emptyset$ and $ J_1 \cap J_2 = \emptyset,$ it implies $I_1={{I_2}^c}=J_2$ and $J_1={{J_2}^c}=I_2,$ then $x=z_2$ and $Ax+q=z_1$ will give the solution of LCP$(q, A),$ otherwise we get nontrivial solution of LCP$(q, A)$ which is not same as $z_1.$ Therefore the homotopy solution $y$ can not give the LCP solution $x$ when ${{I_2}^c}\nsubseteq J_2$ and ${{J_2}^c}\nsubseteq I_2,$ that is ${{I_2}^c}\subseteq J_1$ and ${{J_2}^c}\subseteq I_1.$

First we show that the smooth curve exists for the homotopy function\ref{homf}. 

\begin{theorem}\label{reg}
 Let initial point $y^{(0)} \in \mathcal{F}_1.$ Then $0$ is a regular value of the homotopy function $H:R^{3n} \times (0,1] \to R^{3n}$ and the zero point set $H_{y^{(0)}}^{-1}(0)=\{(y,\lam)\in \mathcal{F}_1:H_{y^{(0)}}(y,\lam)=0\}$ contains a smooth curve $\Gamma_y^{(0)}$ starting from $(y^{(0)}, 1).$  
\end{theorem}

\begin{proof}
	The Jacobian matrix of the above homotopy function $H(y, y^{(0)}, \lambda)$ is denoted by $DH(y,y^{(0)},\lambda)$ and we have $DH(y,y^{(0)}, \lambda)=$$\left[\begin{array}{ccc} 
		\frac{\partial{H(y,y^{(0)},\lambda)}}{\partial{y}} & 	\frac{\partial{H(y,y^{(0)},\lambda)}}{\partial{y^{(0)}}} & \frac{\partial{H(y,y^{(0)},\lambda)}}{\partial{\lam}}\\ 
	\end{array}\right].$ For all $y^{(0)} \in \mathcal{F}_1$ and $\lam \in (0,1],$ we have $\frac{\partial{H(y,y^{(0)},\lambda)}}{\partial{y^{(0)}}}=$$\left[\begin{array}{ccc} -\lam I & 0 & 0\\
	-\lam Z_1^{(0)} & -\lam X^{(0)} & 0\\
	-\lam Z_2^{(0)}A &  0 & -\lam W^{(0)}\\ 
	\end{array}\right],$ where $W^{(0)}=\text{diag}(Ax^{(0)}+q), X^{(0)}=\text{diag}(x^{(0)}),$ $w^{(0)}=Ax^{(0)}+q$ and 
	$\det(\frac{\partial{H}}{ \partial{y^{(0)}}})$$=(-1)^{3n}\lam^{3n}\prod_{i=1}^{n} x_i^{(0)}w_i^{(0)}$ $\neq 0$ for $\lam \in (0,1].$
	Thus $DH(y,y^{(0)},\lambda)$ is of full row rank. Therefore, $0$ is a regular value of $H(y,y^{(0)},\lambda)$ by the Lemma \ref{main}. By Lemmata \ref{par} and \ref{inv}, for almost all  $y^{(0)} \in \mathcal{F}_1,$ $0$ is a regular value of $H_{y^{(0)}}(y,\lam)$ and $H_{y^{(0)}}^{-1}(0)$ consists of some smooth curves and $H_{y^{(0)}}(y^{(0)},1)=0.$ Hence there must be a smooth curve  $\Gamma_y^{(0)}$ starting from  $(y^{(0)},1).$
\end{proof}
Hence by implicit function theorem for every $\lambda$ sufficiently close to $1$, the homotopy function has a unique solution $(y,1) $ of \ref{homf}, which is smooth in the parameter  $\lambda$, in a neighbourhood of $(y^{(0)},1) $.\\

Now we show that the smooth curve $\Gamma_y^{(0)}$ for the homotopy function \ref{homf} is bounded and converges and establish conditions for global convergence of the homotopy method with the homotopy function \ref{homf}. We show that if the $x$ and $z_2$-components of the point $(x,z_1,z_2,\lambda)$ are bounded, then the homotopy curve $\Gamma_y^{(0)}$ is bounded.

\begin{theorem}\label{bnd}
	Let $\mathcal{F}$ be a non-empty set and $A \in R^{n\times n}.$  Assume that there exists a sequence of points $\{u^k\} \subset \Gamma_y^{(0)} \subset \mathcal{F}_1 \times (0,1],$ where $u^k=(x^k,z_1^k,z_2^k, \lam^k)$ such that $\|x^k\|< \infty \ \text{as} \ k \to \infty$ and $\|z_2^k\|< \infty \ \text{as} \ k \to \infty$ and for a given $y^{(0)} \in \mathcal{F}_1,$ $0$ is a regular value of $H(y,y^{(0)},\lambda).$ Then $\Gamma_y^{(0)}$ is a bounded curve in $\mathcal{F}_1 \times (0,1].$  
\end{theorem}

\begin{proof}
Note that $0$ is a regular value of $H(y,y^{(0)},\lambda)$ by Theorem \ref{reg}. Now we assume that $\Gamma_y^{(0)} \subset \mathcal{F}_1 \times (0,1]$ is an unbounded curve. Then there exists a sequence of points $\{u^k\},$ where $u^k=(y^k, \lam^k) \subset \Gamma_y^{(0)}$ such that $\|(y^k, \lam^k)\| \to \infty.$ As $(0,1]$ is a bounded set and $x$ component and $z_2$ component of $\Gamma_y^{(0)}$ is bounded, there exists a subsequence of points  $\{u^k\}=\{(y^k, \lam^k)\}=\{x^k,z_1^k,z_2^k, \lam^k\}$ such that $x^k \to \bar{x},\ {z_2}^k \to \bar{z_2}, \ \lam^k \to \bar{\lam} \in [0,1] \ \text{and} \ \|z^k\| \to \infty \ \text{as} \ k \to \infty, \ \text{where} \ z^k=\left[\begin{array}{c} z_1^k\\ z_2^k\\\end{array}\right].$ Since $\Gamma_y^{(0)} \subset H_{y^{(0)}}^{-1}(0),$ we have
\begin{equation}\label{zzq}
(1-\lambda^k)[(A+A^t)x^k+q-z_1^k-A^tz_2^k]+\lambda^k(x^k-x^{(0)})=0 
\end{equation}	
\begin{equation}\label{yyq}
Z_1^kx^k-\lambda^k Z_1^{(0)}x^{(0)}=0
\end{equation}
\begin{equation}\label{wwwq}
Z_2^k(Ax^k+q)-\lambda^k Z_2^{(0)}(Ax^{(0)}+q)=0
\end{equation}
where $Z_1^k=\text{diag}(z_1^k)$ and $Z_2^k=\text{diag}(z_2^k).$ Let
$\bar{\lam} \in [0,1], \|z_1^k\|=\infty$ and $\|z_2^k\|<\infty$ as $k \to \infty.$
Then $\exists \ i \in \{1,2,\cdots, n\}$ such that $z_{1i}^k \to \infty$ as $k \to \infty.$ Let $I_{1z}=\{i\in\{1,2,\cdots n\} : \lim\limits_{k\to \infty}z_{1i}^k = \infty\}.$ When $\bar{\lam} \in [0,1),$ for $i\in I_{1z}$ we can get from Equation \ref{zzq}, $(1-\lam^k)[((A+A^t)x^k)_i+q_i-z_{1i}^k-(A^tz_{2}^k)_i] + \lam^k(x_i^k-x_i^{(0)})=0$
$\implies (1-\lam^k)z_{1i}^k=(1-\lam^k)[((A+A^t)x^k)_i+q_i-(A^tz_{2}^k)_i]+\lam^k(x_i^k-x_i^{(0)}) \implies z_{1i}^k=[((A+A^t)x^k)_i+q_i-(A^tz_{2}^k)_i]+\frac{\lam^k}{(1-\lam^k)}(x_i^k-x_i^{(0)}).$ As $k \to \infty$ right hand side is bounded, but left hand side is unbounded. It contradicts that $\|z_1^k\|=\infty.$ When $\bar{\lam}=1,$ then from Equation \ref{yyq}, we get, $x_i^k=\frac{\lam^k z_{1i}^{(0)}x_i^{(0)}}{z_{1i}^k}$ for $i \in I_{1z}.$ As $k \to \infty, x_i^k \to 0.$ Again from Equation \ref{zzq}, we obtain	$x_i^{(0)}=\frac{(1-\lambda^k)}{\lambda^k}[((A+A^t)x^k)_i+q_i-z_{1i}^k-(A^tz_2^k)_i]+x_i^k$ for $i \in I_{1z}.$ As $k \to \infty,$ we have  $x_i^{(0)}=-\lim\limits_{k\to \infty}\frac{(1-\lambda^k)}{\lambda^k}z_{1i}^k \leq 0.$ It contradicts that $\|z_1^k\|=\infty.$\\ So  $\Gamma_y^{(0)}$ is a bounded curve in $\mathcal{F}_1 \times (0,1].$  
\end{proof}
Now we show the condition to get bounded curve for nonsingular matrix $A$.
\begin{corol}\label{222}
Let $\mathcal{F}$ be a non-empty set and $A \in R^{n\times n}$  be a nonsingular matrix. Assume that there exists a sequence of points $\{u^k\} \subset \Gamma_y^{(0)} \subset \mathcal{F}_1 \times (0,1],$ where $u^k=(x^k,z_1^k,z_2^k, \lam^k)$ such that $\|x^k\|< \infty \ \text{as} \ k \to \infty.$   Further suppose for $\lam^k \to 1,$ $\|z_2^k\|< \infty \ \text{as} \ k \to \infty.$ Suppose that for a given $y^{(0)} \in \mathcal{F}_1,$ $0$ is a regular value of $H(y,y^{(0)},\lambda).$ Then $\Gamma_y^{(0)}$ is a bounded curve in $\mathcal{F}_1 \times (0,1].$ 
\end{corol}

\begin{proof}
By theorem \ref{reg}, $0$ is a regular value of $H(y,y^{(0)},\lambda)$. Now we assume that $\Gamma_y^{(0)} \subset \mathcal{F}_1 \times (0,1]$ is an unbounded curve. Then there exists a sequence of points $\{u^k\},$ where $u^k=(y^k, \lam^k) \subset \Gamma_y^{(0)}$ such that $\|(y^k, \lam^k)\| \to \infty.$  $(0,1]$ is a bounded set and $x$ component of $\Gamma_y^{(0)}$ is bounded. There exists a subsequence of points  $\{u^k\}=\{(y^k, \lam^k)\}=\{x^k,z_1^k,z_2^k, \lam^k\}$ such that $x^k \to \bar{x},\ $ and suppose for $\lam^k \to 1,$ $\|z_2^k\|< \infty \ \text{as} \ k \to \infty.$ Then two cases will arise.
\vsp
\NI \textbf{Case 1:} $\bar{\lam} \in [0,1], \|z_1^k\|<\infty,$ $\|z_2^k\|=\infty.$\\
Let $\|z_2^k\|=\infty.$ Then $\exists \ j \in \{1,2,\cdots n\}$ such that $z_{2j}^k \to \infty$ as $k \to \infty.$ Let $I_{2z}=\{j\in\{1,2,\cdots n\} : \lim\limits_{k\to \infty}z_{2j}^k = \infty\}.$ When $\bar{\lam} \in [0,1),$ for $j\in I_{2z}$ we can get from Equation \ref{zzq},
$z_{2j}^k=(A^{-t}(A+A^t)x^k)_j+(A^{-t}q)_j-(A^{-t}z_{1}^k)_j+\frac{\lam^k}{1-\lam^k}(x_j^k-x_j^{(0)}).$ As $k \to \infty,$ right hand side is bounded, but left hand side is not. This also contradicts that $\|z_2^k\|=\infty.$ So with our assumption for $\lam^k \to 1,$ $\|z_2^k\|< \infty \ \text{as} \ k \to \infty,$ $\bar{\lam} \in [0,1],$ the homotopy curve is bounded. 
\vsp
\NI \textbf{Case 2:} $\bar{\lam} \in [0,1], \|z_1^k\|=\infty,$ $\|z_2^k\|=\infty.$\\
Let $\|z_1^k\|=\infty, \|z_2^k\|=\infty. $ Then either $\exists \ i \in \{1,2,\cdots n\}$ such that $z_{1i}^k \to \infty$, $z_{2i}^k \to \infty$ as $k \to \infty$ or $\exists \ i,j \in \{1,2,\cdots n\}, i\neq j$ such that $z_{1i}^k \to \infty$ and $z_{2j}^k \to \infty$ as $k \to \infty.$ When $z_{1i}^k \to \infty$, $z_{2i}^k \to \infty$ as $k \to \infty$ and $\bar{\lam} \in [0,1),$ we have,
$z_{1i}^k+(A^tz_{2}^k)_i=((A+A^t)x^k)_i+q_i+\frac{\lam^k}{(1-\lam^k)}(x_i^k-x_i^{(0)}).$ Now as $k \to \infty,$ right hand side is bounded, but left hand side is not, which is impossible. When $\bar{\lam}=1,$ then our assumption $\|z_2^k\|< \infty \ \text{as} \ k \to \infty$ and the argument of the previous theorem \ref{bnd} contradicts that $z_{1i}^k \to \infty,$ $z_{2i}^k \to \infty$ as $k \to \infty.$ As $k \to \infty,$  when $z_{1i}^k \to \infty,$ $z_{2j}^k \to \infty$ for $i \neq j$ as $k \to \infty$ then considering the $i$th and $j$th component and using same argument similar to the previous theorem \ref{bnd} and case 1,  we will get a contradiction.

Thus $\Gamma_y^{(0)}$ is a bounded curve in $\mathcal{F}_1 \times (0,1].$  
\end{proof}

Now we show the necessary condition of the homotopy curve $\Gamma_y^{(0)}$ to be bounded.
\begin{theorem}\label{001}
	Suppose the solution set $\Gamma_y^{(0)}$ of the homotopy function $H(y,y^{(0)},\lambda)=0$ is unbounded. Then there exists $(\xi, \eta, \zeta) \in R_+^{3n}$ such that $e^t \xi =1,$  $\xi^tA\xi \leq 0.$
\end{theorem}
\begin{proof}
	Assume that the solution set $\Gamma_y^{(0)}$ is unbounded. Then there exists a sequence of points $\{u^k\} \subset \Gamma_y^{(0)} \subset \mathcal{F}_1 \times (0,1],$ where $u^k=(x^k,z_1^k,z_2^k, \lam^k)$ such that  $\lim_{k\to \infty}\lam^k=\bar{\lam}$ and either $\|z_2^k\|<\infty$ as $k \to \infty$ with two cases (i) $\lim_{k\to \infty}e^tx^k=\infty$  and	(ii) $\lim_{k\to \infty}(1-\lam^k)e^tx^k= \infty$ or $\lim_{k\to \infty}e^tz_2^k=\infty$ with two cases (i) $\lim_{k\to \infty}e^tx^k=\infty$  and	(ii) $\lim_{k\to \infty}(1-\lam^k)e^tx^k= \infty.$ \\
	First we consider that $\|z_2^k\|<\infty$ as $k \to \infty.$\\
	Case (i) Let $\lim_{k\to \infty}\frac{x^k}{e^tx^k}=\xi \geq 0 $ and $\lim_{k\to  \infty}\frac{z_1^k}{e^tx^k}=\eta \geq 0. $ So it is clear that $e^t\xi=1.$ Then dividing by $e^tx^k$ and taking limit $k \to \infty $ from equations \ref{zzq},\ref{yyq},\ref{wwwq} we get \begin{eqnarray}
	(1-\bar{\lam})[(A+A^t)\xi - \eta]+\bar{\lam}\xi=0\label{n1}\\
	\xi_i \eta_i=0 \ \forall \ i \label{n2}
		\end{eqnarray}
		From equations \ref{n1} and \ref{n2} we get $\eta=(A+A^t)\xi+\frac{\bar{\lam}}{(1-\bar{\lam})}\xi \implies 0=\xi^t\eta=\xi^t[(A+A^t)\xi+\frac{\bar{\lam}}{(1-\bar{\lam})}\xi]$ for $\bar{\lam} \in [0,1).$ This implies that $\xi^t(A+A^t)\xi=-\frac{\bar{\lam}}{(1-\bar{\lam})}\xi \leq 0$  i.e. $\xi^tA\xi\leq 0.$ Specifically for $\bar{\lam}=0,$ $\xi^tA\xi= 0$ and for $\bar{\lam}\in (0,1),$  $\xi^tA\xi<0.$ For $\bar{\lam}=1$ $\xi=0,$ contradicts that $e^t\xi=1.$ \\
		Case (ii) Let $\lim_{k\to \infty}\frac{(1-\lam^k)x^k}{(1-\lam^k)e^tx^k}=\xi'\geq 0.$ Then $e^t\xi'=1.$  Let  $\lim_{k\to \infty}\frac{z_1^k}{(1-\lam^k)e^tx^k}=\eta' \geq 0.$ Then multiplying the equation \ref{zzq} with $(1-\lam^k)$ and dividing by $(1-\lam^k)e^tx^k$, multiplying the equation \ref{yyq} with $(1-\lam^k)$ and dividing by $((1-\lam^k)e^tx^k)^2$ and  multiplying the equation \ref{wwwq} with $(1-\lam^k)$ dividing by $((1-\lam^k)e^tx^k)^2$ and taking limit $k \to \infty$, we get \\
		\begin{eqnarray}
			(1-\bar{\lam})[(A+A^t)\xi' - (1-\bar{\lam})\eta']+\bar{\lam}\xi'=0\label{n11}\\
		\xi'_i \eta'_i=0 \ \forall \ i \label{n22}
		\end{eqnarray}
	Multiplying $(\xi')^t $ in both sides of equation \ref{n11}, we get $(\xi')^tA\xi'\leq 0$ for $\bar{\lam} \in [0,1).$ Specifically for $\bar{\lam}=0,$ $(\xi')^tA\xi'= 0$ and for $\bar{\lam}\in (0,1),$  $(\xi')^tA\xi'<0.$ For $\bar{\lam}=1,$ $\xi'=0,$ contradicts that $e^t\xi'=1.$ \\
		Later we consider that $\lim_{k\to \infty}e^tz_2^k=\infty.$\\
		Case (i) Let $\lim_{k\to \infty}\frac{x^k}{e^tx^k}=\xi \geq 0, $  $\lim_{k\to  \infty}\frac{z_1^k}{e^tx^k}=\eta \geq 0 $ and $\lim_{k\to  \infty}\frac{z_2^k}{e^tx^k}=\zeta \geq 0.$ It is clear that $e^t\xi=1.$ Then dividing by $e^tx^k$ and taking limit $k \to \infty $ from equation \ref{zzq}, dividing by $(e^tx^k)^2$ and taking limit $k \to \infty $ from equation \ref{yyq}, \ref{wwwq}, we get 
		\begin{eqnarray}
		(1-\bar{\lam})[(A+A^t)\xi-\eta-A^t\zeta]+\bar{\lam}\xi=0 \label{nn1}\\
		\xi_i\eta_i=0 \ \forall \ i \label{nn2}\\
		\zeta_i(A\xi)_i=0 \ \forall \ i \label{nn3}
		\end{eqnarray} 
		 From equation \ref{nn1} we get $\eta+A^t\zeta=(A+A^t)\xi+\frac{\bar{\lam}}{1-\bar{\lam}}\xi$ for $\bar{\lam} \in [0,1).$ Now multiplying $\xi^t$ in both sides we get $\xi^t(A+A^t)\xi+\frac{\bar{\lam}}{1-\bar{\lam}}\xi^t\xi=0.$ Hence $\xi^t(A+A^t)\xi=-\frac{\bar{\lam}}{1-\bar{\lam}}\xi^t\xi \leq 0$ for $\bar{\lam} \in [0,1).$ Specifically for $\bar{\lam}=0,$ $\xi^tA\xi= 0$ and for $\bar{\lam}\in (0,1),$  $\xi^tA\xi<0.$ For $\bar{\lam}=1,$ $\xi=0,$ contradicts that $e^t\xi=1.$\\
		 Case(ii) Let $\lim_{k\to \infty}\frac{(1-\lam^k)x^k}{(1-\lam^k)e^tx^k}=\xi'\geq 0.$ Then $e^t\xi'=1.$  Let  $\lim_{k\to \infty}\frac{z_1^k}{(1-\lam^k)e^tx^k}=\eta' \geq 0$ and $\lim_{k\to \infty}\frac{z_2^k}{(1-\lam^k)e^tx^k}=\zeta' \geq 0$  Then multiplying the equation \ref{zzq} with $(1-\lam^k)$ and dividing by $(1-\lam^k)e^tx^k$, multiplying the equation \ref{yyq} with $(1-\lam^k)$ and dividing by $((1-\lam^k)e^tx^k)^2$ and  multiplying the equation \ref{wwwq} with $(1-\lam^k)$ dividing by $((1-\lam^k)e^tx^k)^2$ and taking limit $k \to \infty,$ we get 
		 \begin{eqnarray}
		 (1-\bar{\lam})(A+A^t)\xi'- (1-\bar{\lam})^2\eta'- (1-\bar{\lam})^2A^t\zeta'+\bar{\lam}\xi'=0 \label{nnn1}\\
		 \xi'_i\eta'_i=0 \ \forall \ i \label{nnn2}\\
		 \zeta'_i(A\xi')_i=0 \ \forall \ i \label{nnn3}
		 \end{eqnarray}
		 Multiplying $(\xi')^t$ in both side of equation \ref{nnn1} we get $(\xi')^t(A+A^t)\xi'- (1-\bar{\lam})(\xi')^t\eta'- (1-\bar{\lam})(\xi')^tA^t\zeta'=-\frac{\bar{\lam}}{(1-\bar{\lam})}(\xi')^t\xi'\leq 0$ for $\bar{\lam} \in [0,1).$ Specifically for $\bar{\lam}=0,$ $(\xi')^tA\xi'= 0$ and for $\bar{\lam}\in (0,1),$  $(\xi')^tA\xi'<0.$ For $\bar{\lam}=1, \xi'=0,$ contradicts that $e^t\xi'=1.$
		\end{proof}
	\begin{remk}\label{00}
		Therefore in the neighbouhood of $\bar{\lam}=1$ the homotopy curve is bounded and for the parameter $\lam=0,$ $(\xi)^tA\xi= 0$ and for $\lam\in (0,1),$  $(\xi)^tA\xi<0,$ where $\xi \geq 0,$ $e^t\xi=1.$
	\end{remk}
\begin{corol}
	Suppose $A \in R^{n\times n}$ is a nonsingular matrix and assume that there exists a sequence of points $\{u^k\} \subset \Gamma_y^{(0)} \subset \mathcal{F}_1 \times (0,1],$ where $u^k=(x^k,z_1^k,z_2^k, \lam^k)$ and $\|x^k\|< \infty \ \text{as} \ k \to \infty.$   For a given $y^{(0)} \in \mathcal{F}_1,$ $0$ is a regular value of $H(y,y^{(0)},\lambda).$ Then $\Gamma_y^{(0)}$ is a bounded curve in $\mathcal{F}_1 \times (0,1].$ 
\end{corol}

\begin{theorem}\label{01001}
	Let $A\in R^{n\times n}$ and the set  $\mathcal{F}_1$ be nonempty. For a given $y^{(0)} \in \mathcal{F}_1,$ $0$ is a regular value of $H(y,y^{(0)},\lambda).$  Then the homotopy path $\Gamma_y^{(0)} \subset \mathcal{F}_1 \times (0,1]$ is bounded.
\end{theorem}
\begin{proof}
Suppose $A\in R^{n\times n}$ is a matrix and  there exists a sequence of points $\{u^k\} \subset \Gamma_y^{(0)} \subset \mathcal{F}_1 \times (0,1],$ where $u^k=(x^k,z_1^k,z_2^k, \lam^k).$ Hence by the definition of $\mathcal{F}_1$  $x^k,z_1^k,z_2^k,\\Ax^k+q>0.$ From remark \ref{00} the homotopy curve is bounded in the neighbourhood of $\lam=1.$ Assume that the homotopy curve $\Gamma_y^{(0)} \subset \mathcal{F}_1 \times (0,1)$ is unbounded. Then from theorem \ref{001} ,  $(\xi)^tA\xi<0$ for $\lam\in (0,1).$ But $Ax^k+q>0$ implies that $A\xi \geq 0,$ where $\xi=$$\lim_{k\to \infty}\frac{x^k}{e^tx^k} \geq 0,$ when $\lim_{k\to \infty}{e^tx^k}=\infty$ or $\xi=$$\lim_{k\to \infty}\frac{(1-\lam^k)x^k}{(1-\lam^k)e^tx^k}\geq 0,$ when $\lim_{k\to \infty}{(1-\lam^k)e^tx^k}=\infty.$ Hence $\xi, A\xi \geq 0$ imply that $\xi^tA\xi\geq 0$ for $\lambda\in(0,1)$, which contradicts that the homotopy path is unbounded for $\lam\in(0,1).$ Therefore the homotopy curve $\Gamma_y^{(0)} \subset \mathcal{F}_1 \times (0,1]$ is bounded.
\end{proof}
Hence it is proved that the homotopy curve $\Gamma_y^{(0)}$ is bounded for any matrix $A$.
\begin{theorem}
	For $y^{(0)}=(x^{(0)},z_1^{(0)},z_2^{(0)})\in \mathcal{F}_1,$ the homotopy equation finds a bounded smooth curve $\Gamma_y^{(0)} \subset \mathcal{F}_1 \times (0,1]$ which starts from $(y^{(0)},1)$ and approaches the hyperplane at $\lam =0.$ As $\lam \to 0,$ the limit set $L \times \{0\} \subset \bar{\mathcal{F}}_1 \times \{0\}$ of $\Gamma_y^{(0)}$ is nonempty and every point in $L$ is a solution of the following system:
	\begin{equation}\label{sys}
	\begin{split}
	(A+A^t)x+q-z_1-A^tz_2=0 \\
	Z_1x=0  \\
	Z_2(Ax+q)=0. \\
	\end{split}
    \end{equation}
\end{theorem}
\begin{proof}
  	Note that $\Gamma_y^{(0)}$ is diffeomorphic to a unit circle or a unit interval $(0,1]$ in view of Lemma \ref{cl}. As $\frac{\partial{H(y,y^{(0)},1)}}{\partial{y^{(0)}}}$ is nonsingular, $\Gamma_y^{(0)}$ is diffeomorphic to a unit interval $(0,1].$ Again $\Gamma_y^{(0)}$ is a bounded smooth curve by the Theorem \ref{01001}. Let $(\bar{y},\bar{\lam})$ be a limit point of $\Gamma_y^{(0)}.$ We consider four cases:
	\begin{description}
		\item[Case 1:] $(\bar{y},\bar{\lam})\in \mathcal{F}_1 \times \{1\}.$
		\item[Case 2:] $(\bar{y},\bar{\lam})\in \partial{\mathcal{F}_1} \times \{1\}.$
		\item[Case 3:] $(\bar{y},\bar{\lam})\in \partial{\mathcal{F}_1} \times (0,1).$
		\item[Case 4:] $(\bar{y},\bar{\lam})\in \bar{\mathcal{F}}_1 \times \{0\}.$
	\end{description}
	
As the equation $H_{y^{(0)}}(y,1)=0$ has only one solution $y^{(0)}\in  \mathcal{F}_1, $ the case $1$ is impossible. In case $2$ and $3,$ there exists a subsequence of $(y^k, \lam^k) \in \Gamma_y^{(0)}$ such that $x_i^k \to 0$ or $(Ax^k+q)_i \to 0$ for $i \subseteq \{1,2,\cdots n\}.$ From the last two equalities of the homotopy function \ref{homf}, we have $z_1^k \to \infty$ or $z_2^k \to \infty.$ Hence it contradicts the boundedness of the homotopy path by the Theorem \ref{01001}. Therefore case $4$ is the only possible option. Hence $\bar{y}=(\bar{x},\bar{z_1},\bar{z_2})$ is a solution of the system $(A+A^t)x+q-z_1-A^tz_2=0,  \  Z_1x=0,  \ Z_2(Ax+q)=0.$
\end{proof}

\begin{remk}\label{me}
	From the homotopy function \ref{homf}, we obtain 
	$\bar{z}_{1i}\bar{x}_i = 0$ and $\bar{z}_{2i}(A\bar{x}+q)_i=0 \ \forall i\in\{1,2, \cdots n\}.$ Now $\bar{z}_1$ and $\bar{z}_2$ can be decomposed as $\bar{z}_1= \bar{w}-\Delta\bar{w} \geq 0$  and $\bar{z}_2= \bar{x}-\Delta\bar{x} \geq 0,$ where $\bar{w}=A\bar{x}+q.$ It is clear that $\bar{w}_i \bar{x}_i=\Delta\bar{w}_i \bar{x}_i=\Delta\bar{x}_i\bar{w}_i \ \forall i \in \{1,2, \cdots n\} .$ 
\end{remk}

We demonstrate the condition under which the homotopy functions will give the solution of LCP$(q, A).$
\begin{theorem}
 The component $\bar{x}$ of $(\bar{x},\bar{z}_1,\bar{z}_2,0) \in L\times \{0\}$ gives the solution of LCP$(q, A)$ if and only if $\Delta\bar{x}_i \Delta\bar{w}_i=0 $ or $\bar{z}_{1i}+\bar{z}_{2i}>0 \ \forall i.$
\end{theorem} 

\begin{proof}
Suppose $\bar{x} \geq 0$ and $\bar{w}=A\bar{x}+q \geq 0$ are the solution of LCP$(q, A).$ Then $\bar{x}_i\bar{w}_i=0$ \ $ \forall i.$ This implies that $\bar{x}_i=0$ or $\bar{w}_i=0$ \ $ \forall i.$ We consider the following three cases:
\vsp
\NI \textbf{Case 1:} For at least one $i \in  \{1,2,\cdots n\},$ let $\bar{w}_{i}>0, \bar{x}_{i}=0.$  In view of Remark \ref{me}, this implies that $\Delta\bar{x}_i=0 \implies \Delta\bar{x}_i \Delta\bar{w}_i=0.$
\vsp
\NI	\textbf{Case 2:} For at least one $i \in  \{1,2,\cdots n\},$ let $\bar{x}_{i}>0, \bar{w}_{i}=0.$ In view of \ref{me}, this implies that $\Delta\bar{w}_i=0 \implies \Delta\bar{x}_i \Delta\bar{w}_i=0.$
\vsp
\NI	\textbf{Case 3:}	For at least one $i \in  \{1,2,\cdots n\},$ let $\bar{w}_{i}=0, \bar{x}_{i}=0.$ This implies that either $\Delta\bar{w}_i\Delta\bar{x}_i=0$ or $\bar{z}_{1i}+\bar{z}_{2i}>0.$  
\vsp
For the converse part, consider $\Delta\bar{x}_i \Delta\bar{w}_i=0 $ or $\bar{z}_{1i}+\bar{z}_{2i}>0 \ \forall i \in \{1,2,\cdots n\}.$ Let $\forall i \in \{1,2,\cdots n\}, \ \Delta\bar{x}_i \Delta\bar{w}_i=0 $ implies either $\Delta\bar{x}_i=0$ or $\Delta\bar{w}_i=0.$ This implies that $\bar{w}_i \bar{x}_i=0 \  \forall i \in \{1,2,\cdots n\}.$ Therefore $\bar{w}$ and $\bar{x}$ are the solution of the LCP$(q, A).$ Consider $\bar{z}_{1i}+\bar{z}_{2i}>0 \ \forall  i \in  \{1,2,\cdots n\}.$ Then following three cases will arise.
\vsp
\NI	\textbf{Case 1:}	Let $\bar{z}_{1i}>0, \bar{z}_{2i}=0$ for at least one $i \in  \{1,2,\cdots n\}.$ This implies that $\bar{x}_i=0$ and $\bar{w}_i \geq 0.$
\vsp
\NI \textbf{Case 2:} Let $\bar{z}_{1i}=0, \bar{z}_{2i}>0$ for at least one $i \in  \{1,2,\cdots n\}.$ This implies that $\bar{x}_i \geq 0$ and $\bar{w}_i=0.$
\vsp
\NI \textbf{Case 3:}  Let $\bar{z}_{1i}>0, \bar{z}_{2i}>0$ for at least one $i \in  \{1,2,\cdots n\}.$ This implies that $\bar{x}_i=0$ and $\bar{w}_i=0.$ 
\vsp
Considering the above three cases $\bar{x}, \bar{w}$ solve the LCP$(q, A).$	
\end{proof}
\begin{theorem}
 If $A$ is a $P_0$ matrix, then the component $\bar{x}$ of $(\bar{x},\bar{z}_1,\bar{z}_2,0) \in L\times \{0\}$ gives the solution of LCP$(q, A)$. 
\end{theorem}
\begin{proof}
Let $A$ be a $P_0$ matrix. Assume that the component $\bar{x}$ of $(\bar{x},\bar{z}_1,\bar{z}_2,0) \in L\times \{0\}$ does not give the solution of LCP$(q, A)$. Hence $\Delta\bar{x}_i \Delta\bar{w}_i\neq0 $ and $\bar{z}_{1i}+\bar{z}_{2i}=0$ for atleast one $i$. Then $\Delta\bar{x}_i \neq 0, \Delta\bar{w}_i\neq0 , \bar{z}_{1i}=0, \bar{z}_{2i}=0.$  Now $\bar{z}_{1i}=\bar{w}_i-\Delta\bar{w}_i=0$ and $\Delta\bar{x}_i \Delta\bar{w}_i\neq0$ $\implies \bar{w}_i=\Delta\bar{w}_i>0$. In similar way $\bar{z}_{2i}=\bar{x}_i-\Delta\bar{x}_i=0$ and $\Delta\bar{x}_i \Delta\bar{w}_i\neq0$ $\implies \bar{x}_i=\Delta\bar{x}_i>0$. From Equation \ref{sys}, $\Delta\bar{w}_i + (A^t\Delta\bar{x})_i=0$. This implies that $ (A^t\Delta\bar{x})_i<0$ and also $(\bar{x})_i (A^t\Delta\bar{x})_i<0.$ This contradicts that $A$ is a $P_0$-matrix. Therefore the component $\bar{x}$ of $(\bar{x},\bar{z}_1,\bar{z}_2,0) \in L\times \{0\}$ gives the solution of LCP$(q, A)$.
\end{proof}
\begin{theorem}\label{22222}
Suppose the matrix $(\bar{W}+\bar{X}A^t)$ is nonsingular, where $\bar{W}=\text{diag}(\bar{w}),$ $\bar{X}=\text{diag}(\bar{x}).$ Then $\bar{x}$ solves the LCP$(q, A).$		 	
\end{theorem}
\begin{proof}
Let the matrix $(\bar{W}+\bar{X}A^t)$ be nonsingular. By Equation \ref{sys}, $\Delta\bar{w} + A^t\Delta\bar{x}=0$ and $\bar{X}\Delta\bar{w}=\bar{W}\Delta\bar{x},$ where $\bar{W}=$diag$(\bar{w})=$diag$(A\bar{x}+q).$ Now $\bar{X}\Delta\bar{w} + \bar{X}A^t\Delta\bar{x}=0$ implies that $\bar{W}\Delta\bar{x} + \bar{X}A^t\Delta\bar{x}=0.$ It implies that $\Delta\bar{x}=0.$ Then  $\bar{x}$ solves the LCP$(q,A).$		 	
\end{proof} 
Now we establish a sufficient condition of homotopy method for finding the solution of LCP$(q,A).$ 
\begin{theorem}
    If the matrix $A$ is nondegenerate, then the component $\bar{x}$ of $(\bar{x},\bar{z}_1,\bar{z}_2,0) \in L\times \{0\}$ solves LCP$(q, A)$.
\end{theorem}
\begin{proof}
    Consider that the matrix $A$ associated with LCP$(q,A)$ is nondegenerate. Therefore every principal minor of $A$ is nonzero. By theorem \ref{22222}, if the matrix $(\bar{W}+\bar{X}A^t)$ is nonsingular, then $\bar{x}$ solves the LCP$(q, A)$, where $\bar{W}=\text{diag}(\bar{w}),$ $\bar{X}=\text{diag}(\bar{x}).$  Let $\tilde{\mathcal{A}}=\left[\begin{array}{cc} 
\bar{W} & \bar{X}\\
-A^t & I\\
\end{array}\right]$.  Then $\det(\tilde{\mathcal{A}})= \det(\bar{W}+\bar{X}A^t)$.  Assume that the component $\bar{x}$ of $(\bar{x},\bar{z}_1,\bar{z}_2,0) \in L\times \{0\}$ is not the solution of LCP$(q, A)$. Then there exists atleast one $i$, such that $\bar{x}_i\bar{w}_i>0$. Without loss of generality  $\bar{w}$ and $\bar{x}$ can be represented as $\bar{w}=\left[\begin{array}{c} 
\bar{w}_p \\
\bar{w}_q\\
\bar{o}_r\\
\end{array}\right]$, $\bar{x}=\left[\begin{array}{c} 
\bar{o}_p \\
\bar{x}_q\\
\bar{x}_r\\
\end{array}\right]$, where $\bar{w}_p\in {R^p}_{++}, \ \bar{w}_q, \bar{x}_q \in {R^q}_{++}, \ \bar{x}_r\in {R^r}_{++}$, \ $\bar{o}_r\in R^r, \ \bar{o}_p \in R^p$ and $\bar{o}_r=0, \ \bar{o}_p=0. $ Here $(\bar{w}_q)_i(\bar{x}_q)_i>0$ and $\bar{W}=\text{diag}(\bar{w}), \bar{X}=\text{diag}(\bar{x})$.  Now we can rewrite $\left[\begin{array}{cc} 
\bar{W} & \bar{X}\\
-A^t & I\\
\end{array}\right]=\left[\begin{array}{cccccc} 
\bar{W}_p & \bar{O}_q & \bar{O}_r & \bar{O}_p & \bar{O}_q & \bar{O}_r \\
\bar{O}_p & \bar{W}_q & \bar{O}_r & \bar{O}_p & \bar{X}_q & \bar{O}_r\\
\bar{O}_p & \bar{O}_q & \bar{O}_r & \bar{O}_p & \bar{O}_q & \bar{X}_r\\
M & B & C & I_p & \bar{O}_q & \bar{O}_r\\
D & E & F & \bar{O}_p & I_q & \bar{O}_r\\
G & H & K & \bar{O}_p & \bar{O}_q & I_r\\
\end{array}\right]$, where $-A^t=\left[\begin{array}{ccc} 
M & B & C\\
D & E & F \\
G & H & K \\
\end{array}\right]$, \ $\bar{W}=\left[\begin{array}{ccc} 
\bar{W}_p & \bar{O}_q & \bar{O}_r \\
\bar{O}_p & \bar{W}_q & \bar{O}_r \\
\bar{O}_p & \bar{O}_q & \bar{O}_r \\
\end{array}\right]$, \ $\bar{X}=\left[\begin{array}{ccc} 
 \bar{O}_p & \bar{O}_q & \bar{O}_r \\
 \bar{O}_p & \bar{X}_q & \bar{O}_r\\
 \bar{O}_p & \bar{O}_q & \bar{X}_r\\
\end{array}\right]$, \ $\bar{X}_q=\text{diag}(\bar{x}_q)$, \ $\bar{X}_r=\text{diag}(\bar{x}_r)$, \ $\bar{W}_q=\text{diag}(\bar{w}_q)$, $\bar{W}_p=\text{diag}(\bar{w}_p)$, \ $\bar{O}_p=\text{diag}(\bar{o}_p)$, \ $\bar{O}_q=\text{diag}(\bar{o}_q)$ \ $\bar{O}_r=\text{diag}(\bar{o}_r)$, \ $M,D,G, I_p\in R^{p \times p}$, \ $B,E,H, I_q \in R^{q \times q}$, \ $C,F,K, I_r \in R^{r \times r}$ and $I_p,I_q,I_r$ are identity matrices.  By elementary row operations we can get
\vsp
$\tilde{\mathcal{B}}=\left[\begin{array}{cccccc} 
I & \bar{O}_q & \bar{O}_r & \bar{O}_p & \bar{O}_q & \bar{O}_r \\
\bar{O}_p & I & \bar{O}_r & \bar{O}_p & \bar{X}_q{\bar{W}_q}^{-1} & \bar{O}_r\\
\bar{O}_p & \bar{O}_q & \bar{O}_r & \bar{O}_p & \bar{O}_q & I\\
M & B & C & I & \bar{O}_q & \bar{O}_r\\
D & E & F & \bar{O}_p & I & \bar{O}_r\\
G & H & K & \bar{O}_p & \bar{O}_q & I\\
\end{array}\right]$.
\vsp
By interchanging rows this matrix reduces to \\ $\tilde{\mathcal{C}}=$ $\left[\begin{array}{cccccc} 
I & \bar{O}_q & \bar{O}_r & \bar{O}_p & \bar{O}_q & \bar{O}_r \\
\bar{O}_p & I & \bar{O}_r & \bar{O}_p & \bar{X}_q{\bar{W}_q}^{-1} & \bar{O}_r\\
-G & -H & -K & \bar{O}_p & \bar{O}_q & \bar{O}_r\\
M & B & C & I & \bar{O}_q & \bar{O}_r\\
D & E & F & \bar{O}_p & I & \bar{O}_r\\
G & H & K & \bar{O}_p & \bar{O}_q & I\\
\end{array}\right]$.\\ Hence $\det(\tilde{\mathcal{A}})= \det(\tilde{\mathcal{C}})=(-1)^r\det(K)\neq 0$. Therefore by theorem \ref{22222}, \ $\bar{x}$ solves LCP$(q, A)$. This contradicts the assumption. Hence the component $\bar{x}$ of $(\bar{x},\bar{z}_1,\bar{z}_2,0) \in L\times \{0\}$ is the solution of LCP$(q, A)$. 
\end{proof}
Hence for the $P_0$ and nondegenerate matrix classes the homotopy function \ref{homf} gives the solution of LCP$(q,A)$.
\begin{remk}
   We trace the homotopy path $\Gamma_y^{(0)} \subset \mathcal{F}_1 \times (0,1]$ from the initial point $(y^{(0)},1)$ as $\lambda \to 0.$ To find the solution of the given LCP$(q, A)$ we consider homotopy path along with other assumptions. Let $s$ denote the arc length of $\Gamma_y^{(0)}.$ We parameterize the homotopy path $\Gamma_y^{(0)}$ with respect to $s$ in the following form\\
 	\begin{equation}\label{ss}
 	H_{y^{(0)}} (y(s),\lambda (s))=0, \ 
 	y(0)=y^{(0)}, \   \lambda(0)=1.
 	\end{equation} 
 	Differentiating \ref{ss} with respect to $s,$ we obtain the following system of ordinary differential equations with given initial values 
 	\begin{equation}
     H'_{y^{(0)}} (y(s),\lambda (s))\left[\begin{array}{c} 
      \frac{dy}{ds}\\
     \frac{d\lambda}{ds}\\
     \end{array}\right]=0, \
     \|( \frac{dy}{ds},\frac{d\lambda}{ds})\|=1, \ 
     y(0)=y^{(0)}, \   \lambda(0)=1, \ \frac{d\lambda}{ds}(0)<0, 
    \end{equation} 
    and the $y$-component of $(y(\bar{s}),\lambda (\bar{s}))$ gives the solution of LCP$(q,A)$ for $\lambda (\bar{s})=0.$ For details, see \cite{fan}. 
    
    Note that the parameter $\lambda$ is updated from the Moore-Penrose inverse of the Jacobian matrix for tracing the homotopy path. However, this approach does not ensure that the updated value of the parameter $\lambda$ is in $(0, 1].$ Value of $\lambda$ beyond $(0, 1]$ leads to a non-homotopy path. To eliminate deviation, we propose a modification by introducing a method called \textit{ensuring feasibility by changing step length.} In this method it is necessary to check whether $0 < (\tilde{\lambda}_i - \hat{\lambda}_i) < 1$ and $(\tilde{y}^{i} - \hat{y}^{i}) \in \bar{\cal{F}}_1$ holds or not. If any of the above-mentioned criteria fails, then the step length will be changed appropriately using geometric series to trace the homotopy path $\Gamma_y^{(0)}.$ This guarantees a homotopy continuation trajectory. 
  \end{remk}

\subsection{Algorithm}
\NI \textbf{Step 0:} Initialize $(y^{(0)},\lambda_0).$ Set $l_0 \in (0, 1).$ Choose $\epsilon_2 >> \epsilon_3 >> \epsilon_1 > 0$ which are very small positive quantity.
\vsp 
\NI \textbf{Step 1:} $\tau^{(0)}= \xi^{(0)}=(\frac{1}{n})\left[\begin{array}{c} 
s\\
-1\\
\end{array}\right]$ for $i=0,$ where $n=\|\left[\begin{array}{c} 
s\\
-1\\
\end{array}\right]\|$ and $s= (\frac{\partial H}{\partial y}(y^{(0)},\lambda_0))^{-1}(\frac{\partial H}{\partial \lambda}(y^{(0)},\lambda_0)).$ If $\det  (\frac{\partial H}{\partial y}(y^{(i)},\lambda_i))>0,$ $\tau^{(i)}= \xi^{(i)}$ else $\tau^{(i)}= -\xi^{(i)},$ $i \geq 1.$ Set $l=0.$
\vsp
\NI \textbf{Step 2:} (Predictor point calculation) $(\tilde{y}^{(i)},\tilde{\lambda}_i)=(y^{(i)},\lambda_i)+a\tau^{(i)},$ where $a={l_0}^l.$ Compute $(\hat{y}^{(i)},\hat{\lambda}_{i})=H'_{y^{(0)}}(\tilde{y}^{(i)},\tilde{\lambda}_i)^+ H(\tilde{y}^{(i)}, \tilde{\lambda}_i).$ If $0<(\tilde{\lambda}_i - \hat{\lambda}_{i})<1, $ go to Step 3. Otherwise if $m = \min(a,\|(\tilde{y}^{(i)},\tilde{\lambda}_i)-(\hat{y}^{(i)},\hat{\lambda}_{i})-(y^{(i)},\lambda_i)\|)>a_0,$ update $l$ by $l+1,$ and recompute $(\tilde{\lambda}_i, \hat{\lambda}_{i})$ else go to Step 4.
\vsp
\NI \textbf{Step 3:} (Corrector point calculation)      $(y^{(i+1)},\lambda_{i+1})=(\tilde{y}^{(i)},\tilde{\lambda}_i)-(\hat{y}^{(i)},\hat{\lambda}_{i}).$ Determine the norm $r=\|H(y^{(i+1)},\lambda_{i+1})\|.$ If $r \leq 1$ and $y^{(i+1)}>0$ go to Step 5, otherwise if $a > \epsilon_3,$ update $l$ by $l+1$ and go to Step 2 else go to Step 4.
\vsp
\NI \textbf{Step 4:} If $|\lambda_{i+1} - \lambda_i| < \epsilon_2,$ then if $|\lambda_{i+1}| < \epsilon_2,$ then stop with the solution $(y^{(i+1)},\lambda_{i+1}),$ else terminate (unable to find solution) else $i=i+1$ and go to Step 1.
\vsp
\NI \textbf{Step 5:} If $|\lambda_{i+1}| \leq \epsilon_1,$ then stop with solution $(y^{(i+1)},\lambda_{i+1}),$ else $i=i+1$ and go to Step 1.
\vsp   
Note that in Step 2, $H'_{y^{(0)}}(y,\lambda)^+ = H'_{y^{(0)}}(y,{\lambda})^{t}(H'_{y^{(0)}}(y,{\lambda})H'_{y^{(0)}}(y,\lambda)^{t})^{-1}$ is the Moore-Penrose inverse of $H'_{y^{(0)}}(y,\lambda).$ We prove the following result to obtain the positive direction of the proposed algorithm.
\begin{theorem} \label{direction}
If the homotopy curve $\Gamma_y^{(0)}$ is smooth, then the positive predictor direction $\tau^{(0)}$ at the initial point $y^{(0)}$ satisfies $\det \left[\begin{array}{c}
\frac{\partial H}{\partial y \partial \lam}(y^{(0)},1)\\
\tau ^{(0)^t}\\
\end{array}\right]$$ < 0.$ 
\end{theorem}

\begin{proof}
	From the Equation \ref{homf}, we consider the following homotopy function \\ 
	\vsp
	$H(y,y^{(0)},\lambda)=$ $\left[\begin{array}{c} 
	(1-\lambda)[(A+A^t)x+q-z_1-A^tz_2]+\lambda(x-x^{(0)}) \\
	Z_1x-\lambda Z_1^{(0)}x^{(0)}\\
	Z_2(Ax+q)-\lambda Z_2^{(0)}(Ax^{(0)}+q)\\
	\end{array}\right]=0.$ Now, \\ 
	\vsp
	$\frac{\partial H}{\partial y \partial \lam}(y,\lam)=\left[\begin{array}{cccc} 
	(1-\lambda)(A + A^t) + \lambda I & -(1-\lambda)I & -(1-\lambda)A^t & P\\
	Z_1 & X & 0 & -Z_1^{(0)}x^{(0)}\\
	Z_2A & 0 & W & -Z_2^{(0)}(Ax^{(0)}+q)\\ \
	\end{array}\right],$ where $P = (x-x^{(0)})-[(A+A^t)x+q-z_1-A^tz_2]$ and $W=\text{diag}(Ax+q).$ At the initial point $(y^{(0)},1)$\\
	 $\frac{\partial H}{\partial y \partial \lam}(y^{(0)},1)=\left[\begin{array}{cccc} 
	 I & 0 & 0 & -[(A+A^t)x^{(0)}+q-z^{(0)}_1-A^tz^{(0)}_2]\\
	Z^{(0)}_1 & X^{(0)} & 0 & -Z^{(0)}_1x^{(0)}\\
	Z^{(0)}_2A & 0 & W^{(0)} & -Z^{(0)}_2(Ax^{(0)}+q)\\ 
	\end{array}\right].$\\
	\vsp
	\NI Let positive predictor direction be $\tau^{(0)}=\left[\begin{array}{c}
	\kappa \\ -1
	\end{array}\right] = \left[\begin{array}{c}
	(R^{(0)}_1)^{(-1)}R_2^{(0)} \\ -1
	\end{array} \right],$ where 
	\vsp
	$R^{(0)}_1=\left[\begin{array}{ccc} 
	I & 0 & 0 \\
	Z^{(0)}_1 & X^{(0)} & 0 \\
	Z^{(0)}_2A & 0 & W^{(0)} \\ 
	\end{array}\right],$ $R^{(0)}_2=\left[\begin{array}{c} 
	 -[(A+A^t)x^{(0)}+q-z^{(0)}_1-A^tz^{(0)}_2]\\
	-Z^{(0)}_1x^{(0)} \\
	-Z^{(0)}_2(Ax^{(0)}+q) \\ 
	\end{array}\right]$ and $\kappa$ is a $n \times 1$ column vector.\\
	 Hence, 
    $\det\left[\begin{array}{c}
	\frac{\partial H}{\partial y \partial \lam}(y^{(0)},1)\\
	\tau ^{(0)^t}\\
	\end{array}\right]$\\	
	$=\det\left[\begin{array}{cc}
	R^{(0)}_1 & R^{(0)}_2\\
	(R^{(0)}_2)^t(R^{(0)}_1)^{(-t)} & -1\\	
	\end{array}\right]$ \\ $= \det\left[\begin{array}{cc}
	R^{(0)}_1 & R^{(0)}_2\\
	0 & -1-(R^{(0)}_2)^t(R^{(0)}_1)^{(-t)}(R^{(0)}_1)^{(-1)}R_2^{(0)} \\	\end{array}\right] \\$ 
	\vsp
	\NI $=\det(R^{(0)}_1) \det(-1-(R^{(0)}_2)^t(R^{(0)}_1)^{(-t)}(R^{(0)}_1)^{(-1)}R_2^{(0)})$ \\ 
	\vsp
	\NI $=-\det(R^{(0)}_1) \det(1+(R^{(0)}_2)^t(R^{(0)}_1)^{(-t)}(R^{(0)}_1)^{(-1)}R_2^{(0)})$ \\ 
	\vsp
	\NI $=-\prod_{i=1}^{n}x^{(0)}_i y^{(0)}_i \det(1+(R^{(0)}_2)^t(R^{(0)}_1)^{(-t)}(R^{(0)}_1)^{(-1)}R_2^{(0)}) <0. $ \\
	So the positive predictor direction $\tau ^{(0)}$ at the initial point $y^{(0)}$ satisfies\\ $\det\left[\begin{array}{c}
	\frac{\partial H}{\partial y \partial \lam}(y^{(0)},1)\\
	\tau ^{(0)^t}\\
	\end{array}\right]<0.$
\end{proof}
\vsp
\begin{remk}
	We conclude from the Theorem \ref{direction} that the positive tangent direction $\tau$ of the homotopy path $\Gamma_y^{(0)}$ at any point $(y,\lambda)$ be negative and it depends on det$(R_1),$ where $R_1=\left[\begin{array}{ccc} 
	(1-\lambda)(A+A^t)+\lambda I & -(1-\lambda)I & -(1-\lambda)A^t \\
	Z_1 & X & 0 \\
	Z_2A & 0 & W\\ 
	\end{array}\right].$
\end{remk}

\section{Numerical Examples}
In this section we consider some examples of LCP$(q, A)$ based on $P_0$ and nondegenerate matrices to demonstrate the effectiveness of our proposed algorithm.  Note that Example \ref{matrix2} - \ref{mtx5} are not processable by the algorithms given in Yu et al. \cite{yu2006combined}, Xu et al. \cite{xuuu}, Zhao et al. \cite{N}. Even these examples are not processable by Lemke's algorithm \cite{das2017finiteness} except example  \ref{matrix1} and \ref{matrix2}. Example \ref{matx1} - \ref{mtx5} are also not processable by modulus based algorithm \cite{schafer2004modulus}. We show that the proposed algorithm can process these examples to find the solution.
\begin{examp} \label{matrix0}
		Consider $A=\left[\begin{array}{cc}
		-1 & 2\\
		3  & -1\\
		\end{array}\right]$ and $q=\left[\begin{array}{c}
		1\\
		-0.5\\
		\end{array}\right].$ Note that $A$ is an $N$-matrix. It is solvable by the homotopy method with the homotopy function \ref{zhaon}, proposed by Zhao et al. \cite{N}. Now we show that the homotopy function \ref{homf} also solves the linear complementarity problem with $N$-matrix.  Now choose the initial point $x^{(0)}=\left[\begin{array}{c}
		0.4\\
		0.1\\
		\end{array}\right], \ {z_1}^{(0)}=\left[\begin{array}{c}
		1\\
		1\\
		\end{array}\right]$ and ${z_2}^{(0)}=\left[\begin{array}{c}
		1\\
		1\\
		\end{array}\right].$ Using the proposed algorithm we get the optimal solution of the homotopy function \ref{homf} after 20 iterations and the solution is given by $(\bar{y}, \bar{\lam})=(1,0,0,2.5,1,0,0).$ Therefore $\bar{x}=\left[\begin{array}{c}
		1\\
		0\\
		\end{array}\right]$ solves LCP$(q, A).$ The homotopy path shown in Figure \ref{fig0} illustrates the convergence with respect to the solution vector $x$ and $\lambda$.
	\end{examp}
\begin{examp}\label{matrix1}
	Let $A=\left[\begin{array}{cc}
	1 & -1 \\
	-1  & 1\\
	\end{array}\right]$ and $q=\left[\begin{array}{c}
	-0.5\\
	2\\
	\end{array}\right].$ It is easy to show that $A$ is a $PSD$-matrix. It is solvable by the homotopy method with the homotopy function \ref{psdyu}, proposed by Yu et al.\cite{yu2006combined} . Now we show that the homotopy function \ref{homf} also solves the linear complementarity problem with $PSD$-matrix. Now choose the initial point $x^{(0)}=\left[\begin{array}{c}
	2\\
	1\\
	\end{array}\right],$ ${z_1}^{(0)}=\left[\begin{array}{c}
	1\\
	1\\
	\end{array}\right]$ and ${z_2}^{(0)}=\left[\begin{array}{c}
	1\\
	1\\
	\end{array}\right].$ Using the proposed algorithm we obtain   $(\bar{y},\bar{\lam})=(0.5,0,0,1.499,0.499,0,0,0)$ after 22 iterations. Note that $\bar{x}=\left[\begin{array}{c}
	0.5\\
	0\\
		\end{array}\right]$ is the solution of LCP$(q, A).$ The homotopy path shown in Figure \ref{fig1} illustrates the convergence with respect to the solution vector $x$ and $\lambda$.
\end{examp}
Now we show that the homotopy function \ref{homf} can solve LCP$(q,A)$ with singular matrix $A$ satisfying some conditions.
\begin{examp} \label{matrix2}
	Consider $A=\left[\begin{array}{cc}
	1 & 1\\
	0  & 0\\
	\end{array}\right]$ and $q=\left[\begin{array}{c}
	-1\\
	1\\
	\end{array}\right].$ Note that $A$ is a singular  $Q_0$-matrix. Now choose the initial point $x^{(0)}=\left[\begin{array}{c}
	1\\
	0.2\\
	\end{array}\right], \ {z_1}^{(0)}=\left[\begin{array}{c}
	1\\
	1\\
	\end{array}\right]$ and ${z_2}^{(0)}=\left[\begin{array}{c}
	1\\
	1\\
	\end{array}\right].$ Using the proposed algorithm we get the optimal solution of the homotopy function \ref{homf} after 15 iterations and the solution is given by $(\bar{y}, \bar{\lam})=(1,0,0,1,1,0,0).$ Therefore $\bar{x}=\left[\begin{array}{c}
	1\\
	0\\
	\end{array}\right]$ solves LCP$(q, A).$  The homotopy path shown in Figure \ref{fig2} illustrates the convergence with respect to the solution vector $x$ and $\lambda$.
\end{examp}


\begin{examp} \label{matx1}
	Let $A=\left[\begin{array}{ccc}
	0 & 1 & 1\\
	2  & 0 & 1\\
	-4 & -5 & 0\\
	\end{array}\right]$ and $q=\left[\begin{array}{c}
	-4\\
	-7\\
	10\\
	\end{array}\right].$ It is easy to show that $A$ is an ${E_0}^s$-matrix. This is not processable by modulus based method. Now choose the initial point $x^{(0)}=\left[\begin{array}{c}
	1\\
	1\\
	6\\
	\end{array}\right],$ ${z_1}^{(0)}=\left[\begin{array}{c}
	1\\
	1\\
	1\\
	\end{array}\right]$ and ${z_2}^{(0)}=\left[\begin{array}{c}
	1\\
	1\\
	1\\
	\end{array}\right].$ Using the proposed algorithm we obtain   $(\bar{y},\bar{\lam})=(0,2,7,5,0,0,0,2,7,0)$ after 14 iterations. Note that $\bar{x}=\left[\begin{array}{c}
	0\\
	2\\
	7\\
	\end{array}\right]$ is the solution of LCP$(q, A).$ The convergence of the homotopy function is shown in the Figure \ref{nwfg1}. The first, second and third component of $x$ is represented by data1, data2 and data3 respectively.\\
\end{examp}
\begin{examp} \label{matrix3}
Let $A=\left[\begin{array}{ccc}
-1 & 2 & 1\\
1  & -0.50 & -0.25\\
-0.50 & -1 & -1\\
\end{array}\right]$ and $q=\left[\begin{array}{c}
-0.25\\
-0.10\\
3\\
\end{array}\right].$ It is easy to show that $A$ is not an $N$-matrix. This matrix is not processable by using existing homotopy functions as well as lemke's algorithm.  Now choose the initial point $x^{(0)}=\left[\begin{array}{c}
2.3\\
1\\
0.7\\
\end{array}\right],$ ${z_1}^{(0)}=\left[\begin{array}{c}
1\\
1\\
1\\
\end{array}\right]$ and ${z_2}^{(0)}=\left[\begin{array}{c}
1\\
1\\
1\\
\end{array}\right].$ Using the proposed algorithm we obtain   $(\bar{y},\bar{\lam})=(1.8333,0,2.0833,0,1,2125,0,1.8333,0,2.0833,0)$ after 17 iterations. Note that $\bar{x}=\left[\begin{array}{c}
1.8333\\
0\\
2.0833\\
\end{array}\right]$ is the solution of LCP$(q, A).$ The convergence of the homotopy function is shown in the Figure \ref{fig3}. The first, second and third component of $x$ is represented by data1, data2 and data3 respectively.
\end{examp}

\begin{examp} \label{matrix4}
	Let $A=\left[\begin{array}{ccc}
	1 & -2 & 0\\
	0  & 1 & -2\\
	-2 & 0 & 1\\
	\end{array}\right]$ and $q=\left[\begin{array}{c}
	-1\\
	1\\
	7\\
	\end{array}\right].$ It is easy to show that $A$ is an almost $C_0$ matrix. This matrix is not processable by lemke's algorithm as well as modulus based algorithm. This matrix is also not processable by existing homotopy methods. Now choose the initial point $x^{(0)}=\left[\begin{array}{c}
	3\\
	0.5\\
	0.5\\
	\end{array}\right], \ {z_1}^{(0)}=\left[\begin{array}{c}
	1\\
	1\\
	1\\
	\end{array}\right]$ and ${z_2}^{(0)}=\left[\begin{array}{c}
	1\\
	1\\
	1\\
	\end{array}\right].$ Using the proposed algorithm we obtain  is $(\bar{y}, \bar{\lam})=(1,0,0,0,1,5,1,0,0,0)$ after $24$ iterations. Note that $\bar{x}=\left[\begin{array}{c}
	1\\
	0\\
	0\\
	\end{array}\right]$ solves LCP$(q, A),$ which is a degenerate solution. The convergence of the homotopy function is shown in the Figure \ref{fig4}. The first, second and third component of $x$ is represented by data1, data2 and data3 respectively.
\end{examp}
\begin{examp} \label{matrix5}
	Let $A=\left[\begin{array}{cccc}
	-1 & 1 & 1 & 1\\
	1  & 0 & 0 & 0\\
	1 & 0 & 0 & -1\\
	1 & 0 & -1 & 0\\
	\end{array}\right]$ and $q=\left[\begin{array}{c}
	-1\\
	1\\
	-1\\
	1\\
	\end{array}\right].$ $A$ is a $Q$-matrix by \cite{neogy2005almost} and also almost $\bar{N}$-matrix. This matrix is not processable by lemke's algorithm. Now choose the initial point $x^{(0)}=\left[\begin{array}{c}
	4\\
	4\\
	1\\
	1\\
	\end{array}\right], \ {z_1}^{(0)}=\left[\begin{array}{c}
	1\\
	1\\
	1\\
	1\\
	\end{array}\right]$ and ${z_2}^{(0)}=\left[\begin{array}{c}
	1\\
	1\\
	1\\
	1\\
	\end{array}\right].$ We apply our proposed algorithm to this LCP$(q, A)$ and after 17 iterations we get the approximate optimal solution of the homotopy function \ref{homf}, which is  $(\bar{y}, \bar{\lam})=(1,0,2,0,0,2,0,0,1,0,2,0,0).$ Note that $\bar{x}=\left[\begin{array}{c}
	1\\
	0\\
	2\\
	0\\
	\end{array}\right]$ solves LCP$(q, A),$ which gives a degenerate solution. The convergence of the homotopy function is shown in the Figure \ref{fig5}. Data1, data2, data3 and data4 represent the first, second, third and fourth component of $x$ respectively.
\end{examp}
\begin{examp} \label{maat4}
	Let $A=\left[\begin{array}{cccc}
	-2 & -2 & -2 & 2\\
	-2  & -1 & -3 & 3\\
	-2 & -3 & -1 & 3\\
	2 & 3 & 3 & 0\\
	\end{array}\right]$ and $q=\left[\begin{array}{c}
	-1001\\
	-500\\
	-500\\
	-500\\
	\end{array}\right].$ $A$ is a almost $N_0$-matrix by \cite{neogy2005almost} but not $Q$-matrix. This matrix is not processable by lemke's algorithm as well as modulus based algorithm. This matrix is also not processable by existing homotopy methods.  Now choose the initial point $x^{(0)}=\left[\begin{array}{c}
	100\\
	100\\
	200\\
	1000\\
	\end{array}\right], \ {z_1}^{(0)}=\left[\begin{array}{c}
	1\\
	1\\
	1\\
	1\\
	\end{array}\right]$ and ${z_2}^{(0)}=\left[\begin{array}{c}
	1\\
	1\\
	1\\
	1\\
	\end{array}\right].$ We apply our proposed algorithm to this LCP$(q, A)$ and after 17 iterations we get the approximate optimal solution of the homotopy function \ref{homf}, which is  $(\bar{y}, \bar{\lam})=(250,0,0,750.50,0,1251.50,1251.50,0,250,0,0,750.50,0).$ Note that $\bar{x}=\left[\begin{array}{c}
	250\\
	0\\
	0\\
	750.50\\
	\end{array}\right]$ solves LCP$(q, A),$ which gives a degenerate solution. The convergence of the homotopy function is shown in the Figure \ref{fiiig4}. Data1, data2, data3 and data4 represent the first, second, third and fourth component of $x$ respectively.
\end{examp}

\begin{examp} \label{matrix6}
	Consider $A=\left[\begin{array}{ccccc}
	0 & 0 & 0 & 1 & 2\\
	0  & 0 & -1 & -1 & 2\\
	0 & -1 & 0 & -1 & 1\\
	1 & -1 & -1 & 0 & 0\\
	2 & 1 & 0 & 0 & 0\\
	\end{array}\right]$ and $q=\left[\begin{array}{c}
	-2\\
	-1\\
	7\\
	2\\
	-1\\
	\end{array}\right].$ $A$ is an $N_0$-matrix of exact order $2.$ This matrix is not processable by lemke's algorithm as well as modulus based algorithm. This matrix is also not processable by existing homotopy methods.  Now choose the initial point $x^{(0)}=\left[\begin{array}{c}
	3\\
	1\\
	1\\
	1\\
	3\\
	\end{array}\right],$ ${z_1}^{(0)}=\left[\begin{array}{c}
	1\\
	1\\
	1\\
	1\\
	1\\
	\end{array}\right]$ and ${z_2}^{(0)}=\left[\begin{array}{c}
	1\\
	1\\
	1\\
	1\\
	1\\
	\end{array}\right].$ Using the proposed algorithm, we obtain the approximate optimal solution of the homotopy function \ref{homf}, $(\bar{y}, \bar{\lam}) = (0.5, 0, 0, 0, 1, 0, 1, 8, 2.5, 0, 0.5, 0, 0, 0, 1, 0)$ after 27 iterations. Note that $\bar{x}=\left[\begin{array}{c}
	0.5\\
	0\\
	0\\
	0\\
	1\\
	\end{array}\right]$ solves LCP$(q, A).$ The convergence of the homotopy function is shown in the Figure \ref{fig6}.  Data1, data2, data3, data4 and data5 represent the first, second, third, fourth and fifth component of $x$ respectively.
	0,-90,-80,-70,0,-90,-2,-2,-2,2,-70,-2,-1,-3,3,-50,-2,-3,-0.8,3,0,2,3,3,0
\end{examp}
\begin{examp}\label{mtx5}
	Consider $A=\left[\begin{array}{ccccc}
0 & -90 & -80 & -70 & 0\\
-90  & -2 & -2 & -2 & 2\\
-70& -2 & -1 & -3 & 3\\
-50 & -2 & -3 & -0.8 & 3\\
0 & 2 & 3 & 3 & 0\\
\end{array}\right]$ and $q=\left[\begin{array}{c}
400\\
50\\
30\\
20\\
-10\\
\end{array}\right].$ $A$ is an $\bar{N}$-matrix of exact order $2.$ This matrix is not processable by lemke's algorithm as well as modulus based algorithm. This matrix is also not processable by existing homotopy methods.  Now choose the initial point $x^{(0)}=\left[\begin{array}{c}
0.1\\
0.1\\
0.1\\
5\\
100\\
\end{array}\right],$ ${z_1}^{(0)}=\left[\begin{array}{c}
1\\
1\\
1\\
1\\
1\\
\end{array}\right]$ and ${z_2}^{(0)}=\left[\begin{array}{c}
1\\
1\\
1\\
1\\
1\\
\end{array}\right].$ Using the proposed algorithm, we obtain the approximate optimal solution of the homotopy function \ref{homf}, $(\bar{y}, \bar{\lam}) = (0.2403846, 0, 1.634615, 3.846154, 0,$

 $0, 17.40385, 0, 0, 6.442308, 0.2403846, 0, 1.634615, 3.846154, 0, 0)$ after 1925 iterations. Note that $\bar{x}=\left[\begin{array}{c}
0.2403846\\
0\\
1.634615\\
3.846154\\
0\\
\end{array}\right]$ solves LCP$(q, A).$ The convergence of the homotopy function is shown in the Figure \ref{fg5}.  Data1, data2, data3, data4 and data5 represent the first, second, third, fourth and fifth component of $x$ respectively.
\end{examp}
\newpage
\begin{figure}[H]
	\begin{subfigure}{.5\textwidth}
		\centering
		\includegraphics[height=1.5in, width=4in]{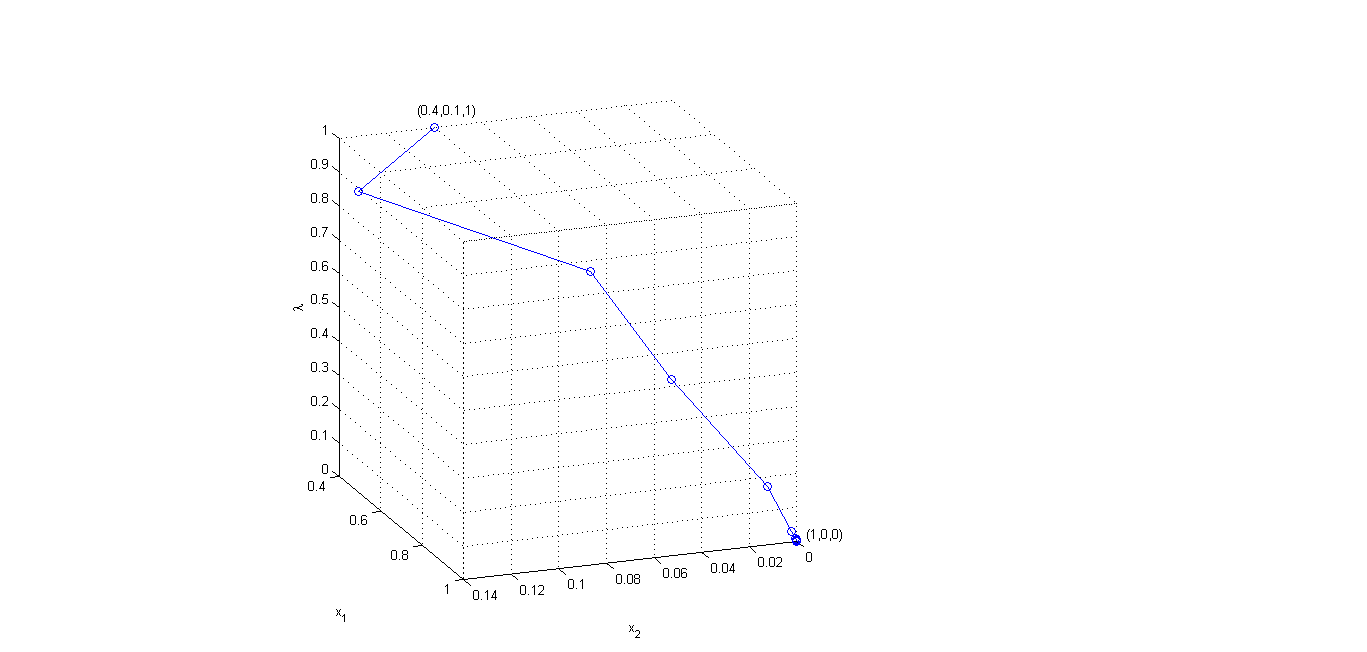}  
		\caption{Example \ref{matrix0}}
		\label{fig0}
	\end{subfigure}	
	\quad
	\begin{subfigure}{.5\textwidth}
		\centering
		\includegraphics[height=1.5in, width=5in]{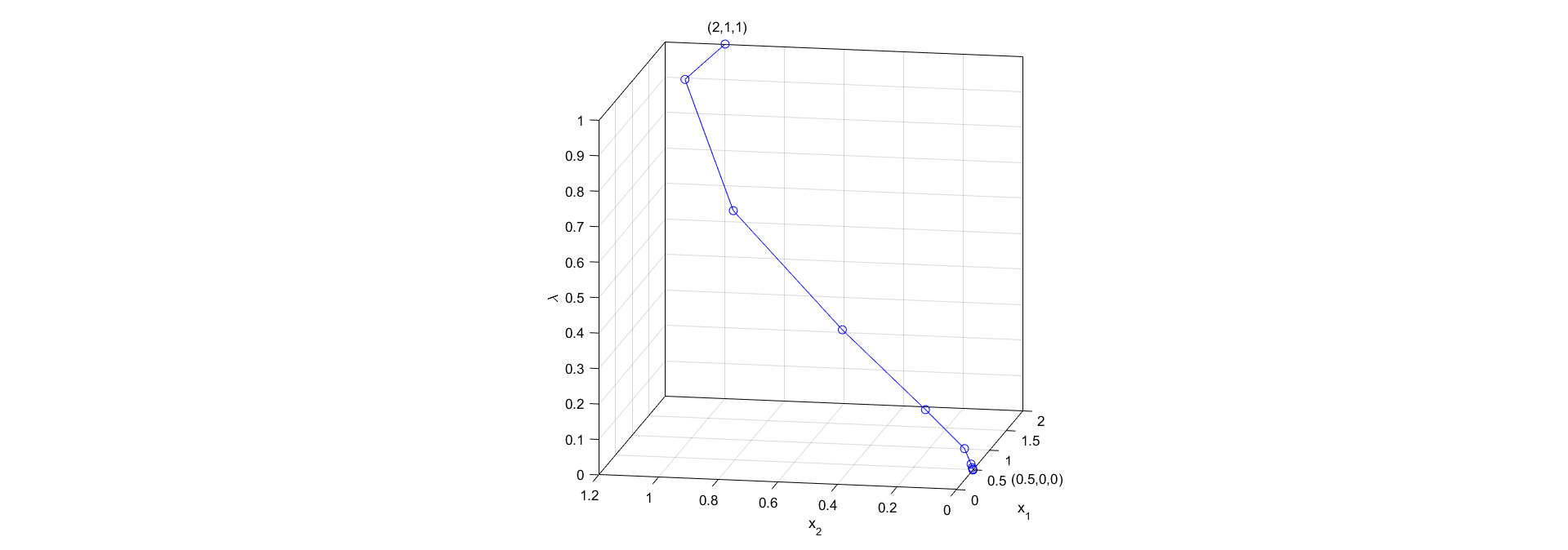}  
		\caption{Example \ref{matrix1}}
		\label{fig1}
	\end{subfigure}
	\quad
	\begin{subfigure}{.5\textwidth}
		\centering
		\includegraphics[height=1.5in, width=4in]{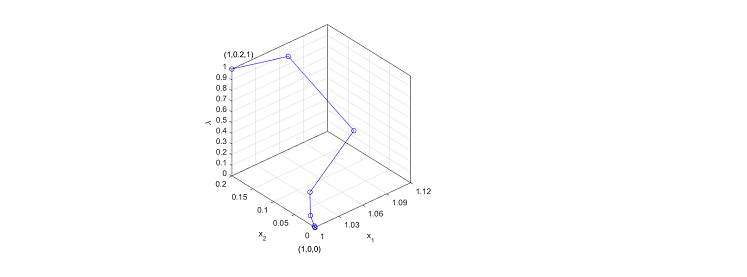}  
		\caption{Example \ref{matrix2}}
		\label{fig2}
	\end{subfigure}
	\quad
	\begin{subfigure}{.5\textwidth}
		\centering
		\includegraphics[height=1.5in, width=4in]{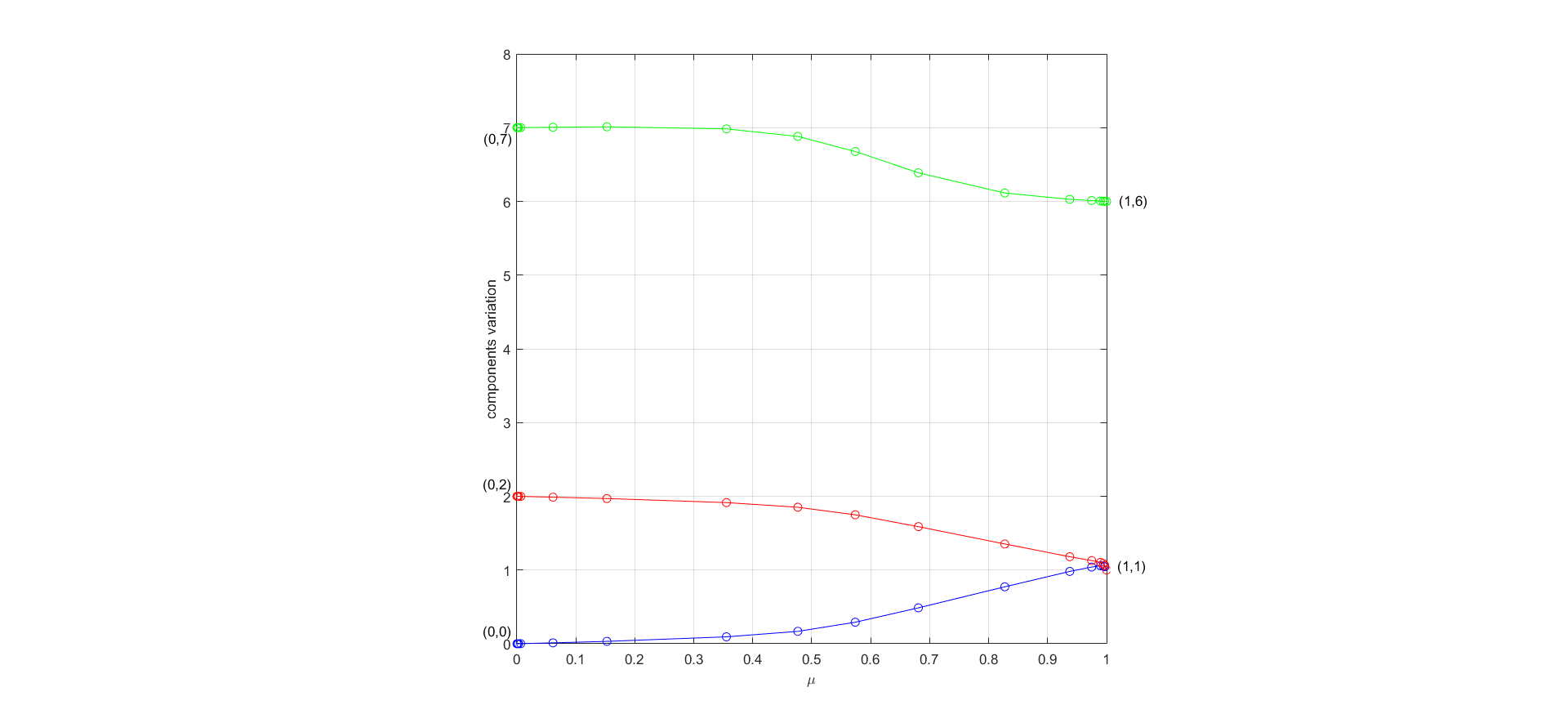}
		\caption{Example \ref{matx1}}
		\label{nwfg1}
	\end{subfigure}
\quad
	\begin{subfigure}{.5\textwidth}
		\centering
		\includegraphics[height=1.5in, width=4in]{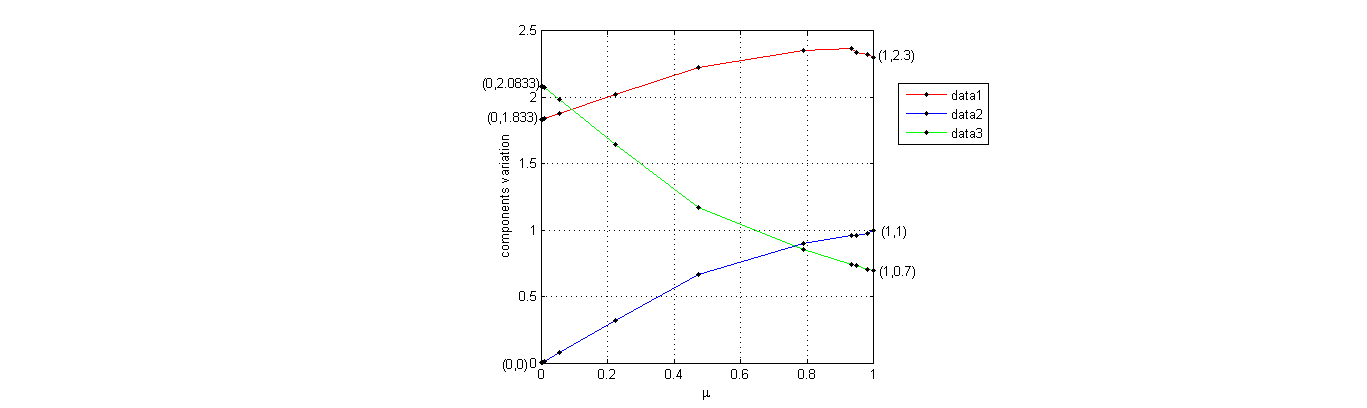}  
		\caption{Example \ref{matrix3}}
		\label{fig3}
	\end{subfigure}
	\quad
	\begin{subfigure}{.5\textwidth}
		\centering
		\includegraphics[height=1.5in, width=4in]{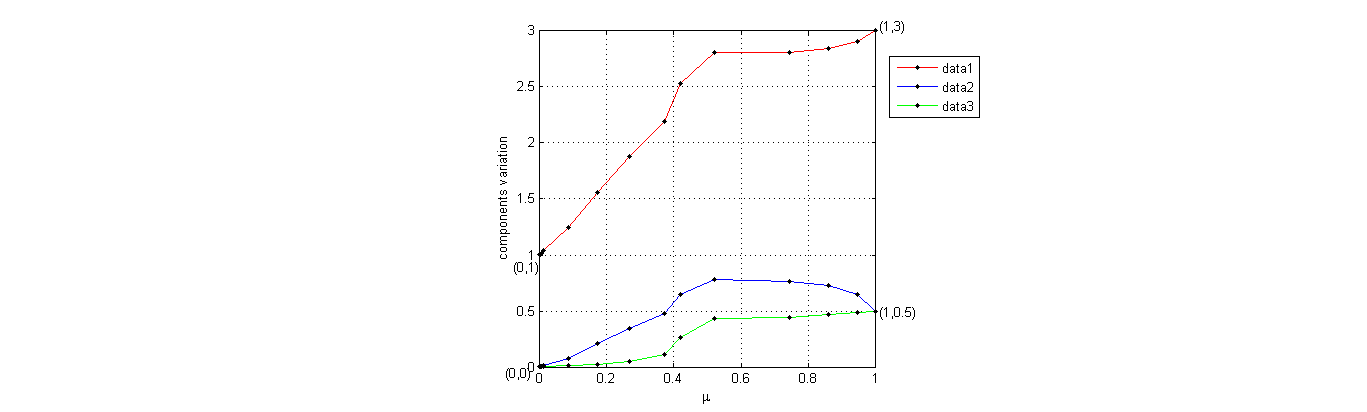}  
		\caption{Example \ref{matrix4}}
		\label{fig4}
	\end{subfigure}
	\quad
	\begin{subfigure}{.5\textwidth}
		\centering
		\includegraphics[height=1.5in, width=4in]{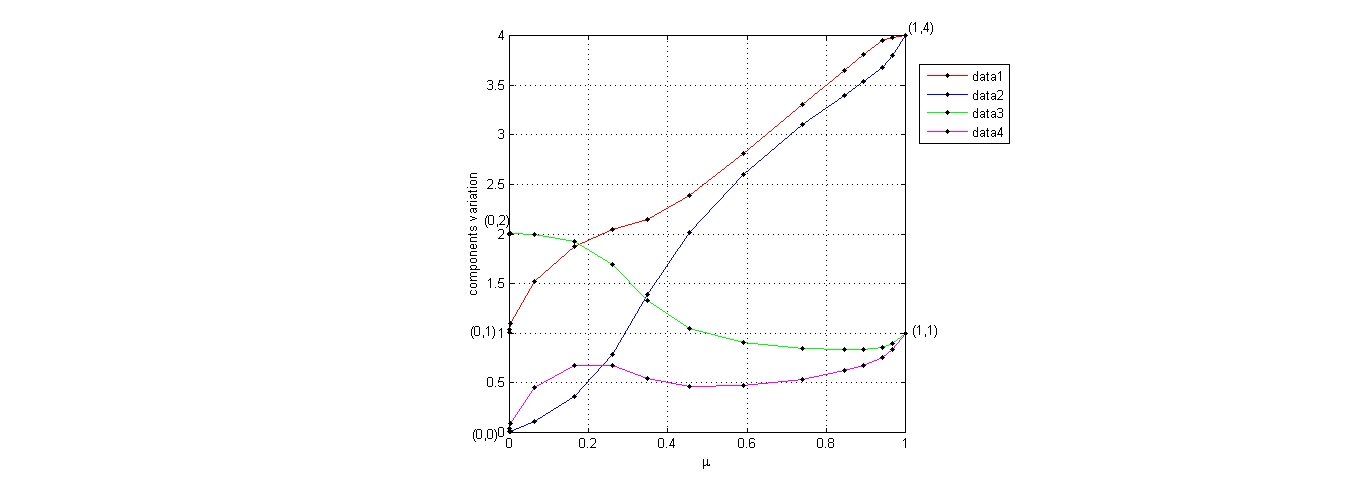}  
		\caption{Example \ref{matrix5}}
		\label{fig5}
	\end{subfigure}
	\quad
	\begin{subfigure}{.5\textwidth}
		\centering
		\includegraphics[height=1.5in, width=4in]{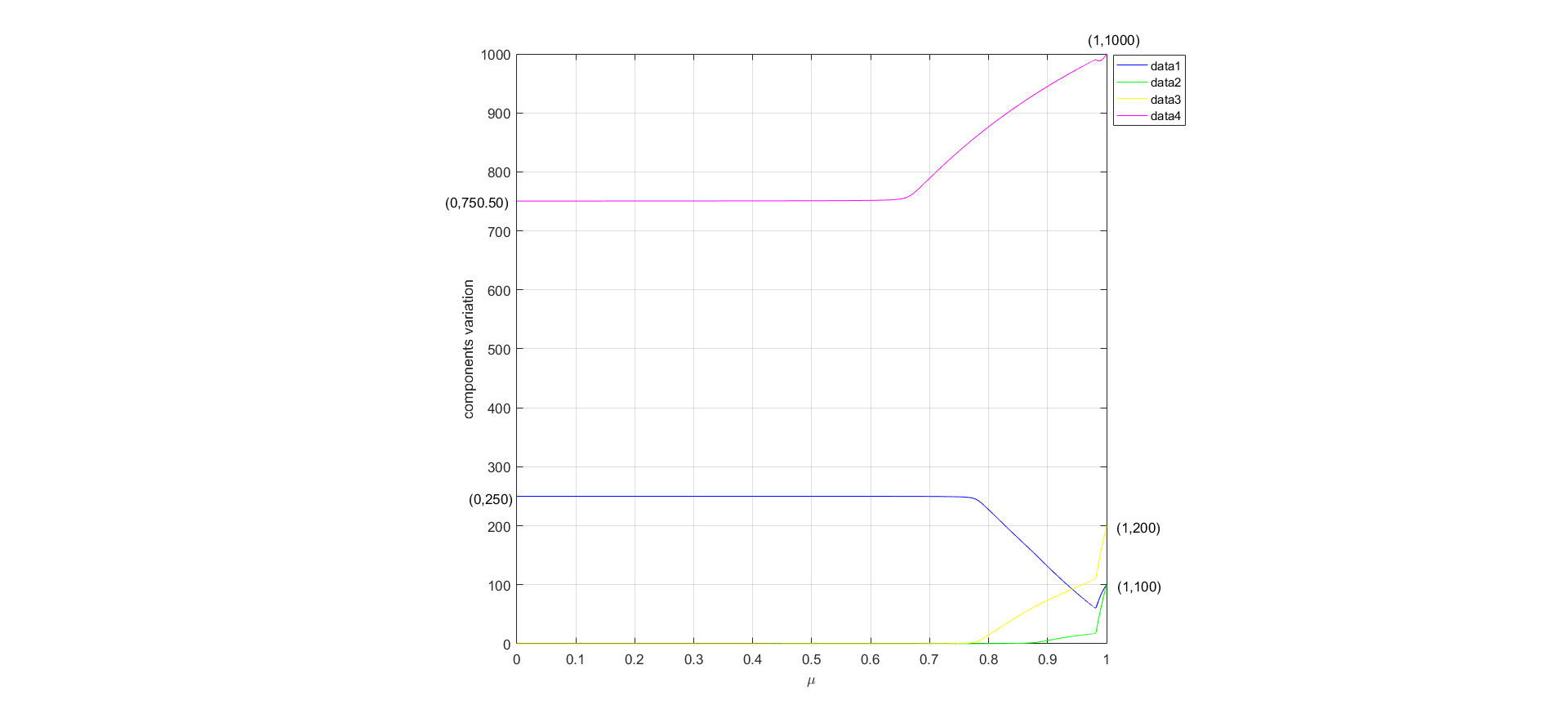}  
		\caption{Example \ref{maat4}}
		\label{fiiig4}
	\end{subfigure}
\quad
	\begin{subfigure}{.5\textwidth}
		\centering
		\includegraphics[height=1.5in, width=4in]{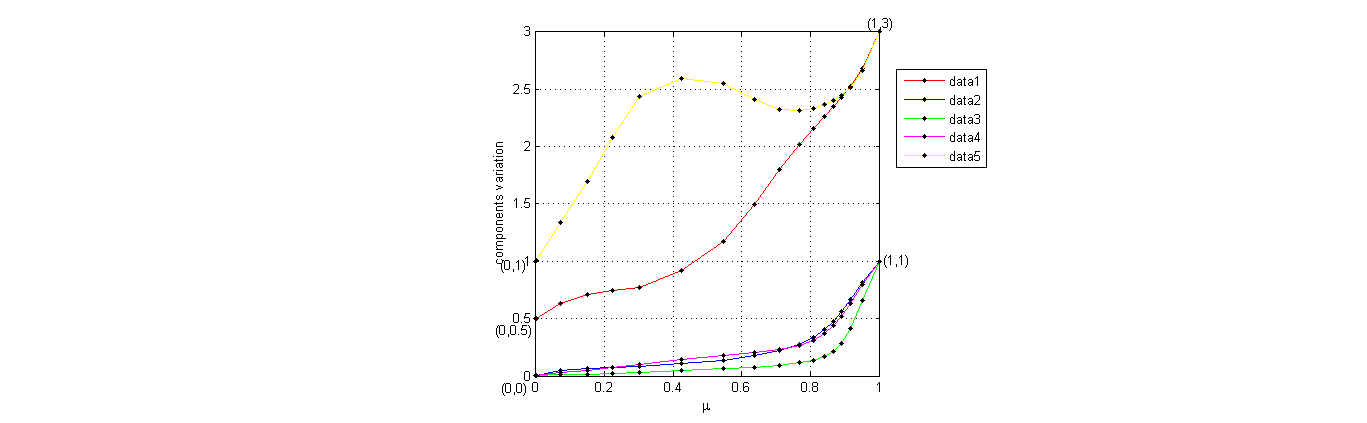}  
		\caption{Example \ref{matrix6}}
		\label{fig6}
	\end{subfigure}
\quad
\begin{subfigure}{.5\textwidth}
	\centering
	\includegraphics[height=1.5in, width=4in]{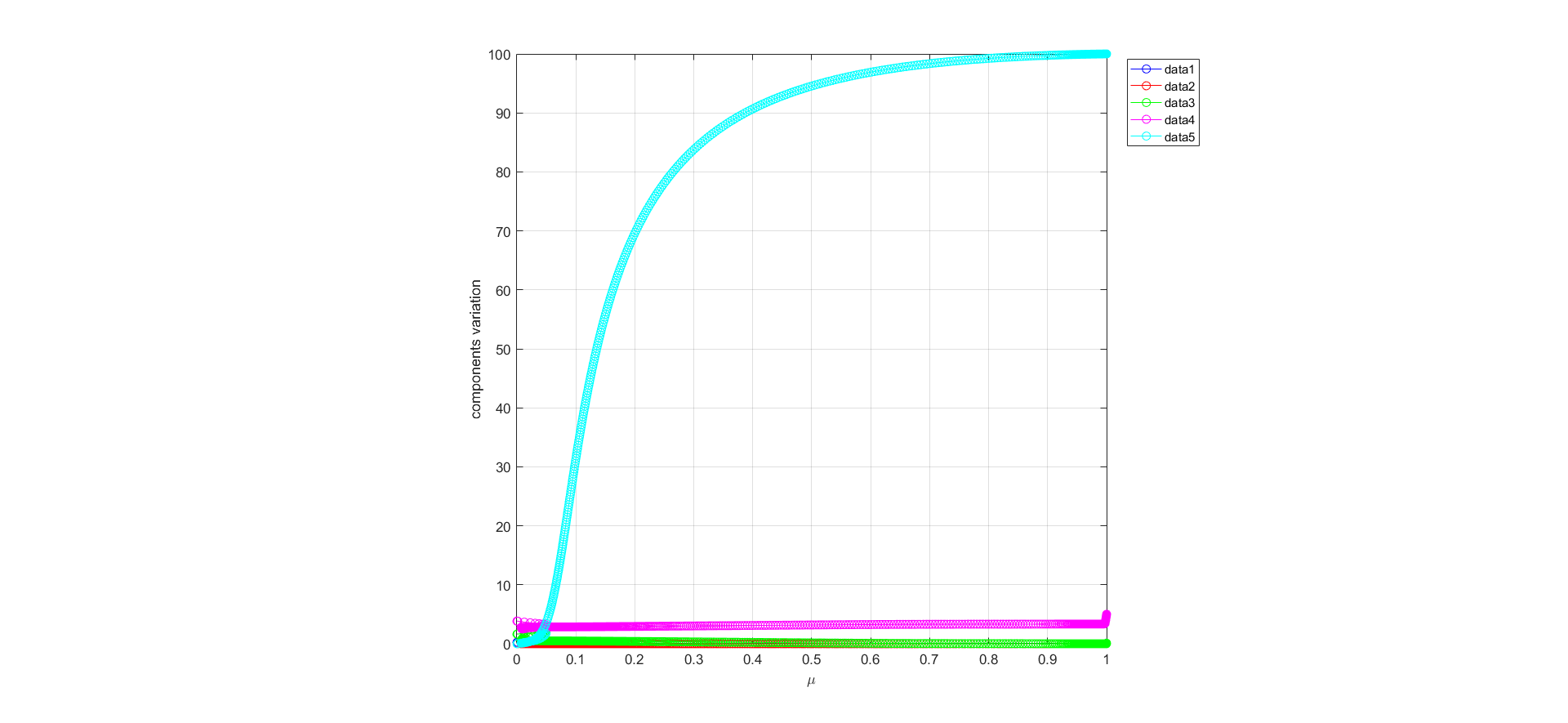}  
	\caption{Example \ref{mtx5}}
	\label{fg5}
\end{subfigure}
	\caption{Homotopy path for the LCP$(q, A)$ to show the convergence}
	\label{fig8}
\end{figure}

\section{Conclusion}
In this study, we consider an interior point homotopy path to solve linear complementarity problem. We prove a necessary and sufficient condition for the solution of LCP$(q,A)$ based on newly introduced homotopy function. To ensure a homotopy continuation trajectory we introduce a new scheme of choosing step length. Mathematically we find the positive tangent direction of the homotopy path. We show that the smooth curve for the homotopy function is bounded and convergent. Several numerical examples are presented to demonstrate the processability of larger classes of LCP$(q, A)$ based on $P_0$ and nondegenerate matrices namely,  $Q$-matrix, almost $\bar{N}$-matrix, $Q_0$-matrix, almost $N_0$-matrix, almost $C_0$-matrix, $N_0$-matrix of exact order $2$ and $\bar{N}$-matrix of exact order $2$. Many of them are not processable by lemke's algoritm, existing homotopy method and  modulus based method. However, the proposed method is able to process all the cases to find solution. 
\section{Acknowledgment}
The author A. Dutta is thankful to the Department of Science and Technology, Govt. of India, INSPIRE Fellowship Scheme for financial support. We acknowledge Mr. Abhirup Ganguly(M.Tech 2017-2019, ISI Kolkata) for his contribution.
\vsp

\bibliographystyle{plain}
\bibliography{homotopy}
\end{document}